\documentclass[12pt, letterpaper]{article}
\usepackage{dirtytalk}
\usepackage{amsmath}
\usepackage{amsfonts}
\usepackage{amssymb}
\usepackage{amsthm}
\usepackage{breqn}
\usepackage{setspace}
\usepackage{fullpage}
\usepackage{enumitem}
\usepackage{comment}
\usepackage{hyperref}
\usepackage{bbm}
\usepackage{tikz}
\usepackage{enumitem}
\usepackage{mathtools} 
\usepackage[utf8]{inputenc}
\usepackage[english]{babel}
\usepackage{mathtools}
\usetikzlibrary{patterns,arrows,decorations.pathreplacing}
\newtheorem{lemma}{Lemma}
\newtheorem{theorem}{Theorem} 
 
\newtheorem{definition}{Definition}
\newtheorem{claim}{Claim}

\newtheorem{conjecture}{Conjecture}

\newcommand{\p}{\mathcal{P}}
\newcommand{\F}{\mathcal{F}}

\DeclareMathOperator{\ex}{ex}

\usepackage{authblk}
\usetikzlibrary{patterns}

\title{Planar Tur\'an Number of the $\Theta_6$}
\author[1]{\hspace{1cm}Debarun Ghosh}
\author[1,2]{Ervin Gy\H{o}ri} 
\author[1,3]{Addisu Paulos}
\author[1]{\newline Chuanqi Xiao}
\author[1,4]{Oscar Zamora}
\affil[1]{Central European University, Budapest\par
\texttt{ghosh\textunderscore debarun@phd.ceu.edu, xiao\textunderscore chuanqi@outlook.com}}
\affil[2]{Alfr\'ed R\'enyi Institute of Mathematics, Budapest \par
\texttt{gyori.ervin@renyi.mta.hu}}
\affil[3]{Addis Ababa University, Addis Ababa\par\texttt{addisu\textunderscore 2004@yahoo.com,addisu.wmeskel@aau.edu.et}}
\affil[4]{Universidad de Costa Rica, San Jos\'e\par
\texttt{oscarz93@yahoo.es}}
\date{}

\begin{document}
\maketitle
\begin{abstract} Let $\mathcal{F}$ be a nonempty family of graphs. A graph $G$ is called $\mathcal{F}$-\textit{free} if it contains no graph from $\mathcal{F}$ as a subgraph. For a positive integer $n$,  the \emph{planar Tur\'an number} of $\F$, denoted by $\ex_{\p}(n,\F)$, is the maximum number of edges in an $n$-vertex $\F$-free planar graph.

Let $\Theta_k$ be the family of Theta graphs on $k\geq 4$ vertices, that is, graphs obtained by joining a pair of non-consecutive vertices of a $k$-cycle with an edge. Lan, Shi and Song 
determined an upper bound  $\text{ex}_{\mathcal{P}}(n,\Theta_6)\leq \frac{18}{7}n-\frac{36}{7}$, but for large $n$, they did not verify that the bound is sharp.
In this paper, we improve their bound by proving $\text{ex}_{\mathcal{P}}(n,\Theta_6)\leq \frac{18}{7}n-\frac{48}{7}$ and then we demonstrate the existence of infinitely many positive integer $n$ and an $n$-vertex $\Theta_6$-free planar graph attaining the bound. 
\end{abstract}
\section{Introduction}
In this paper, all graphs considered are planar, undirected, finite, and contains neither loops nor multiple edges. We use the notations $C_k$ and  $P_k$ to denote a cycle and a path on $k$ vertices respectively. Let $\Theta_k$ denote the family of Theta graphs on $k\geq 4$ vertices, that is, graphs obtained by joining a pair of non-consecutive vertices of a $C_k$ with an edge. We describe a member of $\Theta_k$ as a $\Theta_k$-graph. It can be checked that the family $\Theta_k$ contains $\lfloor{\frac{k}{2}}\rfloor-1$ $\Theta_k$-graphs.

Let $\mathcal{F}$ be a family of graphs. A graph $G$ is called $\mathcal{F}$-\textit{free} if it contains no graph from $\mathcal{F}$ as a subgraph. The case that $\mathcal{F}=\{F\}$, we may say $G$ is $F$-free instead of saying $\mathcal{F}$-free. Accordingly, we shall call a graph $G$ as $\Theta_k$-free if it contains no $\Theta_k$-graph as a subgraph. Moreover, we say $G$ contains $\Theta_k$ if there is a $\Theta_k$-graph contained in $G$ as a subgraph. Notice that a graph $G$ is $\Theta_k$-free if and only if $G$ contains no non-induced $k$-cycle as a subgraph. 

For a nonempty family of graphs $\F$ and a positive integer $n$, the \emph{Tur\'an number} of $\F$, denoted by $\ex(n,\F)$, is the maximum number of edges in an $n$-vertex $\F$-free graph, i.e.,  
\begin{align*}
\ex(n,\F)=\max \{e(G): \mbox{$G$ is an $n$-vertex $\F$-free graph} \}.
\end{align*}
The case that $\F=\{F\}$, we simply denote $\ex(n,\F)$ by $\ex(n,F)$.

Regarding this in 1941, Tur\'an~\cite{turan} proved a classical result in the field of extremal graph theory. He determined exactly the Tur\'an number of any complete graph. A major breakthrough in the study of Tur\'an number of graphs came in 1966, with the proof of the famous theorem by Erd\H{o}s, Stone and Simonovits~\cite{erdos1,erdos2}. They determined an asymptotic value of the Tur\'an number of any non-bipartite graph $F$. In particular, they proved  $\ex(n,F)=\left(1-\frac{1}{\chi(F)-1}\right){n\choose 2}+o(n^2)$, where $\chi(F)$ is the chromatic number of $F$.
Since the two results researchers have been interested working on Tur\'an number of class of bipartite (degenerate) graphs and extremal graph problems with some more generality.

\begin{definition}
Let $\F$ be a nonempty family of graphs and $n$ be a positive integer. The \textbf{planar Tur\'an number of $\F$}, denoted by $\ex_{\p}(n,\F)$, is the maximum number of edges an $n$-vertex $\F$-free planar graph may contain, i.e., 
\begin{align*}
\ex_{\p}(n,\F)=\max \{e(G): \mbox{$G$ is an $n$-vertex $\F$-free planar graph} \}.
\end{align*}
The case that $\F=\{F\}$, we simply denote $\ex_{\p}(n,\F)$ by $\ex_{\p}(n,F)$.
\end{definition}

In 2015, Dowden~\cite{dowden} initiated the study of planar Tur\'an number of graphs. He determined sharp upper bounds of $\ex_{\mathcal{P}}(n,C_4)$ and $\ex_{\mathcal{P}}(n,C_5)$.
\begin{theorem}~\cite{dowden}
\begin{enumerate}
    \item $\ex_{\mathcal{P}}(n,C_4)\leq \frac{15(n-2)}{7}$ , for all $n\geq 4$. 
    \item $\ex_{\mathcal{P}}(n,C_5)\leq \frac{12n-33}{5}$,  for all $n\geq 11$.
\end{enumerate}
\end{theorem}
Extending Dowden's results,  Lan, Shi and Song~\cite{lan} obtained sharp upper bound for $\text{ex}_{\mathcal{P}}(n,\Theta_i)$, for $i\in\{4, 5\}$ and an upper bounds for $\text{ex}_{\mathcal{P}}(n,\Theta_6)$. 
\begin{theorem}\cite{lan} \label{sw}
\begin{enumerate}
    \item $\ex_{\mathcal{P}}(n,\Theta_4)\leq \frac{12(n-2)}{5}$, for all $n\geq 4$.
    \item $\ex_{\mathcal{P}}(n,\Theta_5)\leq \frac{5(n-2)}{2}$, for all $n\geq 5$.
    \item \label{lm3}$\ex_{\mathcal{P}}(n,\Theta_6)\leq \frac{18(n-2)}{7}$, for all $n\geq 6$, with equality when $n=9.$
\end{enumerate}
\end{theorem}
From (\ref{lm3}) of Theorem \ref{sw}, Lan, Shi and Song deduced that $\text{ex}_{\mathcal{P}}(n, C_6)\leq \frac{18(n-2)}{7}$.   Recently, Ghosh, Gy\H{o}ri, Martin, Paulos and Xiao~\cite{ghosh}, improved the bound by giving a sharp upper bound. 

In this paper, we improve the additive constant of the bound of $\text{ex}_{\mathcal{P}}(n,\Theta_6)$ given in Theorem \ref{sw} and illustrate that our bound is sharp. More precisely, we give infinitely many positive $\Theta_6$-free planar graphs attaining the bound. For more results on planar Tur\'an numbers of other graphs, we refer\cite{FW, DG2, DG, LY1}. 
\begin{theorem}\label{vcb}
Let $n\geq 6$. If $G$ is an $n$-vertex $2$-connected  $\Theta_6$-free planar graph with $\delta(G)\geq3$, then $e(G)\leq \frac{18}{7}n-\frac{48}{7}.$
\end{theorem}
\begin{theorem}\label{bvc}
$\ex_{\mathcal{P}}(n,\Theta_6)\leq \frac{18}{7}n-\frac{48}{7}$, for all $n\geq 14$.
\end{theorem}

The following notations are used frequently in the upcoming sections. Let $G$ be a graph. We denote the vertex and the edge sets of $G$ by $V(G)$ and $E(G)$ respectively. The number of vertices and edges in $G$ are respectively denoted by $v(G)$ and $e(G)$.  For a vertex $v$ in $G$, the degree of $v$ is denoted by $d_{G}(v)$. We may omit the subscript if the underlying graph is clear.  The set of all vertices in $G$ which are adjacent to $v$ is denoted as $N(v)$.  We use $\delta(G)$ to denote the minimum degree of $G$. For the sake of simplicity, we use the term $k$-cycle to mean the cycle of length $k$. Similarly for $k$-path, $k$-face, and so on. A describe a path with end vertices $u$ and $v$ as a $(u,v)$-path. 


\section{Extremal Construction}
In this section we shall show the existence of infinitely many positive integers $n$ and $\Theta_6$-free planar graphs $G$ on $n$ vertices such that $e(G)=\frac{18}{7}n-\frac{48}{7}$. This is done based on Dowden's construction (with some modifications) in \cite{dowden}  while proving the bound $\text{ex}_{\mathcal{P}}(n, C_5)\leq \frac{12n-33}{5}$ is tight. The following lemma plays the central role in proving so. 
\begin{lemma}\cite{dowden}\label{gfr}
For infinitely many values of $k$, there exists a plane triangulation $T_k$ with vertex set $\{v_1, v_2,\cdots, v_k\}$ satisfying 
\begin{enumerate}
\item $d(v_i)=4$ for $i\leq 6,$ 
\item $d(v_i)=6$ for $i>6,$
\item $E(T_k)\supset\{v_1v_2,v_3v_4,v_5v_6\}.$
\end{enumerate}
\end{lemma}
For instance the case of $T_{15}$ is as shown in Figure \ref{ccf}. The red edges shown in the figure are those edges indicated in (3) of Lemma \ref{gfr}.  
\begin{figure}[ht]
\centering
\begin{tikzpicture}[scale=0.12]
\draw[thick](-6,20)--(0,26)--(6,20)--(0,14)--(-6,20)--(0,10)--(6,20);
\draw[thick](-6,-20)--(0,-26)--(6,-20)--(0,-14)--(-6,-20)--(0,-10)--(6,-20);
\draw[thick](-6,0)--(0,10)--(6,0)--(0,6)--(-6,0)(-6,0)--(0,-10)--(6,0)--(0,-6)--(-6,0);
\draw[thick](0,14)--(0,10)--(0,6)(0,-14)--(0,-10)--(0,-6);
\draw[thick,red](0,26)--(0,14)(0,6)--(0,-6)(0,-26)--(0,-14);
\draw[thick](-6,20)--(-6,0)--(-6,-20);
\draw[thick](6,20)--(6,0)--(6,-20);
\draw[thick](6,-20)--(0,-30)--(-6,-20)(0,-26)--(0,-30);
\draw[black,thick](6,20)..controls (18,25) and (18,-25) .. (6,-20);
\draw[black,thick](-6,20)..controls (-18,25) and (-18,-25) .. (-6,-20);
\draw[black,thick](6,20)..controls (25,30) and (29,-30) .. (0,-30);
\draw[black,thick](-6,20)..controls (-25,30) and (-29,-30) .. (0,-30);
\draw[black,thick](0,26)..controls (35,35) and (35,-35) .. (0,-30);
\draw[fill=black](0,6)circle(25pt);
\draw[fill=black](0,10)circle(25pt);
\draw[fill=black](0,14)circle(25pt);
\draw[fill=black](0,26)circle(25pt);
\draw[fill=black](6,0)circle(25pt);
\draw[fill=black](6,20)circle(25pt);
\draw[fill=black](-6,0)circle(25pt);
\draw[fill=black](-6,20)circle(25pt);
\draw[fill=black](0,-6)circle(25pt);
\draw[fill=black](0,-10)circle(25pt);
\draw[fill=black](0,-14)circle(25pt);
\draw[fill=black](0,-26)circle(25pt);
\draw[fill=black](6,0)circle(25pt);
\draw[fill=black](6,-20)circle(25pt);
\draw[fill=black](-6,0)circle(25pt);
\draw[fill=black](-6,-20)circle(25pt);
\draw[fill=black](0,-30)circle(25pt);
\end{tikzpicture} 
\caption{An example of Lemma \ref{gfr} with 15 vertices.}
\label{ccf}
\end{figure}
As a consequence of Lemma \ref{gfr}, we have the following theorem. 
\begin{theorem}
There exist infinitely many positive integers $n$ and $\Theta_6$-free planar graphs $G$ on $n$ vertices such that $e(G)=\frac{18}{7}n-\frac{48}{7}.$ 
\end{theorem}
\begin{proof}
We use the triangulation $T_k$ which is obtained from Lemma \ref{gfr} to prove the theorem. See Figure \ref{ccf} for the case of $k=15$. Let $E^*$ denote the set of edges $\{v_1v_2,v_3v_4,v_5v_6\}$ as stated in the lemma. For the example in Figure \ref{ccf}, $E^*$ is the set of the red edges. We construct the base graph $G_k$ (with $4k+12$ vertices) from $T_k$ with the following procedures. 
\begin{enumerate}
\item Subdivide all edges in $E(T_k)\backslash E^*$. Notice that since $T_k$ is a maximal plane graph with $k$ vertices, then $e(T_k)=3k-6$. Thus, the number of subdividing vertices is $3k-9$.
\item  Replace all edges in $E^*$ with the \say{diamond holder} shown in Figure~\ref{jjn}. 
\begin{figure}[ht]
\centering 
\begin{tikzpicture}[scale=0.13]
\draw[thick,blue](10,-28)--(0,-18)--(-10,-28)--(0,-38)--(10,-28);
\draw[thick,blue](5,-28)--(0,-23)--(-5,-28)--(0,-33)--(5,-28);
\draw[thick,blue](0,-18)--(0,-23)(0,-38)--(0,-33)(0,-28)--(5,-28)(0,-28)--(-5,-28)(0,-28)--(0,-23)(0,-28)--(0,-33);
\draw[thick,black](-25,-38)--(-25,-18);
\draw[fill=black](-25,-38)circle(20pt);
\draw[fill=black](-25,-18)circle(20pt);
\draw[fill=black](10,-28)circle(15pt);
\draw[fill=black](-10,-28)circle(15pt);
\draw[fill=black](0,-18)circle(20pt);
\draw[fill=black](0,-38)circle(20pt);
\draw[fill=black](0,-23)circle(15pt);
\draw[fill=black](0,-33)circle(15pt);
\draw[fill=black](5,-28)circle(15pt);
\draw[fill=black](-5,-28)circle(15pt);
\draw[fill=black](0,-28)circle(15pt);
\node at (-18,-28) {$\Longrightarrow$};
\end{tikzpicture}
\caption{Replacing an edge in $E^*$ with a diamond holder}
\label{jjn}
\end{figure}
    
Denote the newly obtained plane graph as $G_k$. The case of $G_{15}$ is shown in Figure \ref{jkl}. Notice that all the $k$ vertices of the $T_k$ in the $G_k$ are degree 6 vertices (including the six vertices that only had degree 4 in $T_k$). For instance,  all the red vertices in $G_{15}$, see Figure \ref{jkl}, are the 15 vertices of $T_{15}$ and each of this vertex is of degree 6 in $G_{15}$. We can consider all the $k$ vertices of $T_k$  as centre of the star $K_{1,6}$ in $G_k$. 
\end{enumerate}

Next, we construct $G$ using the base graph $G_k$ as follows. Replace each star, $K_{1,6}$, with a \say{snow flake} shown in Figure \ref{jjm}, where the central vertex of the star is replaced by a hexagon and the edges are replaced by $K_5^{-}$'s. 
\begin{figure}[ht]
\centering 
\begin{tikzpicture}[scale=0.1]
\draw[thick](10,0)--(8.7,5)--(5,8.7)--(0,10)--(-5,8.7)--(-8.7,5)--(-10,0)(-5,-8.7)--(0,-10)--(5,-8.7);
\draw[thick](-10,0)--(-8.7,-5)--(-5,-8.7);
\draw[thick](10,0)--(8.7,-5)--(5,-8.7);
\draw[thick](5,8.7)--(11.3,6.5)--(10,0)--(13.9,8)--(5,8.7)(8.7,5)--(13.9,8);
\draw[thick](10,0)--(5,8.7)--(-5,8.7)--(-10,0)--(-5,-8.7)--(5,-8.7)--(10,0);
\draw[thick](-5,8.7)--(-11.3,6.5)--(-10,0)--(-13.9,8)--(-5,8.7)(-8.7,5)--(-13.9,8);
\draw[thick](-5,-8.7)--(-11.3,-6.5)--(-10,0)--(-13.9,-8)--(-5,-8.7)(-8.7,-5)--(-13.9,-8);
\draw[thick](5,-8.7)--(11.3,-6.5)--(10,0)--(13.9,-8)--(5,-8.7)(8.7,-5)--(13.9,-8);
\draw[thick](5,8.7)--(0,13)--(-5,8.7)--(0,16)--(5,8.7)(0,10)--(0,16);
\draw[thick](-5,-8.7)--(0,-13)--(5,-8.7)--(0,-16)--(-5,-8.7)(0,-10)--(0,-16);
\draw[thick,black](-40,0)--(-40,-16)(-40,0)--(-40,16)(-40,0)--(-53.9,-8)(-40,0)--(-53.9,8)(-40,0)--(-27,-8)(-40,0)--(-27,8);
\draw[fill=black](-40,-16)circle(30pt);
\draw[fill=black](-40,16)circle(30pt);
\draw[fill=black](-53.9,-8)circle(30pt);
\draw[fill=black](-27,-8)circle(30pt);
\draw[fill=black](-27,8)circle(30pt);
\draw[fill=black](-53.9,8)circle(30pt);
\draw[fill=red](-40,0)circle(30pt);
\draw[fill=black](10,0)circle(12pt);
\draw[fill=black](-10,0)circle(12pt);
\draw[fill=black](0,10)circle(15pt);
\draw[fill=black](0,13)circle(15pt);
\draw[fill=black](0,16)circle(25pt);
\draw[fill=black](0,-10)circle(15pt);
\draw[fill=black](0,-13)circle(15pt);
\draw[fill=black](0,-16)circle(30pt);
\draw[fill=black](8.7,5)circle(15pt);
\draw[fill=black](11.3,6.5)circle(15pt);
\draw[fill=black](-11.3,-6.5)circle(15pt);
\draw[fill=black](11.3,-6.5)circle(15pt);
\draw[fill=black](-13.9,-8)circle(30pt);
\draw[fill=black](13.9,-8)circle(30pt);
\draw[fill=black](13.9,8)circle(30pt);
\draw[fill=black](-8.7,5)circle(15pt);
\draw[fill=black](-11.3,6.5)circle(15pt);
\draw[fill=black](-13.9,8)circle(30pt);
\draw[fill=black](5,8.7)circle(15pt);
\draw[fill=black](-5,8.7)circle(15pt);
\draw[fill=black](-5,-8.7)circle(15pt);
\draw[fill=black](5,-8.7)circle(15pt);
\draw[fill=black](-8.6,-5)circle(15pt);
\draw[fill=black](8.6,-5)circle(15pt);
\node at (-20,0) {$\Longrightarrow$};
\end{tikzpicture}
\caption{Replacing a $K_{1,6}$ with a snowflake}
\label{jjm}
\end{figure}

Let the graph we obtained be $G$ and containing $n$ vertices. Notice that $G$ contains $k$ snowflakes and $3$ diamond holders. 

Now we count the number of vertices of $G$ in terms of $k$. Notice the number of vertices of a snowflake except the tip vertices of the six $K_5^{-}$ is 18. Thus, $G$ has $18k$ such vertices. The remaining vertices of $G$ are the subdividing vertices, which is $3k-9$, and $21$ vertices of the three diamond holders. Thus, $n=18k+(3k-9)+21=21k+12$. This implies, $k=\frac{n-12}{21}$.

Let us now compute the number of edges in $G$. It can be checked that each snowflake contains 54 edges. The remaining edges are the 8 edges that appear in the interior (except the two hanging edges) of each diamond holder. Thus, $e(G)=54k+24$.
Therefore, using the two results we get $e(G)=\frac{18}{7}n-\frac{48}{7}.$

Next we prove that $G$ is a $\Theta_6$-free graph. Observe that except the $3$-faces of the three diamond holders, each face of $G_k$ is of size $6$. Due to the subdivision of edges of $T_k$, all the  $6$-cycles of $G_k$ (except those $6$-cycles  containing an interior vertex of a diamond holder) are boundary of a $6$-face. Therefore, each such  $6$-cycle in $G_k$ is an induced cycle, i.e, no two non consecutive vertices of the cycle are adjacent. On the other hand, there are only two $6$-cycles that contains an interior vertex of a given diamond holder of $G_k$.  Moreover, it can be checked that the cycles are induced and  therefore, every $6$-cycles in $G_k$ are induced, i.e., $G_k$ is a $\Theta_6$-free graph. 

Notice that any $6$-cycle in $G$ either induce
a $6$-cycle of in $G_k$ or must be contained entirely within a $K_5^-$ incident to a snowflake. However, the later case can not happen, as it contains only $5$ vertices. Therefore, every $6$-cycle of $G$ is an induced cycle, and hence $G$ is $\Theta_6$-free graph.
\end{proof}

\begin{figure}[ht]
\centering  
\begin{tikzpicture}[scale=0.17]
\draw[thick](-6,20)--(0,26)--(6,20)--(0,14)--(-6,20)--(0,10)--(6,20);
\draw[thick](-6,-20)--(0,-26)--(6,-20)--(0,-14)--(-6,-20)--(0,-10)--(6,-20);
\draw[thick](-6,0)--(0,10)--(6,0)--(0,6)--(-6,0)(-6,0)--(0,-10)--(6,0)--(0,-6)--(-6,0);
\draw[thick](0,14)--(0,10)--(0,6)(0,-14)--(0,-10)--(0,-6);
\draw[thick,red](0,26)--(0,14)(0,6)--(0,-6)(0,-26)--(0,-14);
\draw[thick](-6,20)--(-6,0)--(-6,-20);
\draw[thick](6,20)--(6,0)--(6,-20);
\draw[black,thick](6,20)..controls (18,25) and (18,-25) .. (6,-20);
\draw[black,thick](-6,20)..controls (-18,25) and (-18,-25) .. (-6,-20);
\draw[black,thick](6,20)..controls (25,30) and (29,-30) .. (0,-30);
\draw[black,thick](-6,20)..controls (-25,30) and (-29,-30) .. (0,-30);
\draw[black,thick](0,26)..controls (35,35) and (35,-35) .. (0,-30);
\draw[thick,blue](0,23)--(3,20)--(0,17)--(-3,20)--(0,23)--(0,20)--(0,17)(-3,20)--(0,20)--(3,20)(0,26)--(0,23)(0,17)--(0,14);
\draw[thick,blue](0,-23)--(3,-20)--(0,-17)--(-3,-20)--(0,-23)--(0,-20)--(0,-17)(-3,-20)--(0,-20)--(3,-20)(0,-26)--(0,-23)(0,-17)--(0,-14);
\draw[thick,blue](-3,0)--(0,3)--(3,0)--(0,-3)--(-3,0)--(0,0)--(3,0)(0,-3)--(0,0)--(0,3)(0,6)--(0,3)(0,-6)--(0,-3);
\draw[thick,blue](-4.5,20)--(0,26)--(4.5,20)--(0,14)--(-4.5,20);
\draw[thick,blue](-4.5,-20)--(0,-26)--(4.5,-20)--(0,-14)--(-4.5,-20);
\draw[thick,blue](-4.5,0)--(0,6)--(4.5,0)--(0,-6)--(-4.5,0);
\draw[thick](6,-20)--(0,-30)--(-6,-20)(0,-26)--(0,-30);
\draw[fill=red](0,6)circle(15pt);
\draw[fill=red](0,10)circle(15pt);
\draw[fill=red](0,14)circle(15pt);
\draw[fill=red](0,26)circle(15pt);
\draw[fill=black](6,0)circle(12pt);
\draw[fill=red](6,20)circle(15pt);
\draw[fill=black](-6,0)circle(12pt);
\draw[fill=red](-6,20)circle(15pt);
\draw[fill=red](0,-6)circle(15pt);
\draw[fill=red](0,-10)circle(15pt);
\draw[fill=red](0,-14)circle(15pt);
\draw[fill=red](0,-26)circle(15pt);
\draw[fill=red](6,0)circle(15pt);
\draw[fill=red](6,-20)circle(15pt);
\draw[fill=red](-6,0)circle(15pt);
\draw[fill=red](-6,-20)circle(15pt);
\draw[fill=black](0,12)circle(12pt);
\draw[fill=black](0,8)circle(12pt);
\draw[fill=black](0,-12)circle(12pt);
\draw[fill=black](0,-8)circle(12pt);
\draw[fill=black](6,10)circle(12pt);
\draw[fill=black](6,-10)circle(12pt);
\draw[fill=black](-6,10)circle(12pt);
\draw[fill=black](-6,-10)circle(12pt);
\draw[fill=black](15,0)circle(12pt);
\draw[fill=black](-15,0)circle(12pt);
\draw[fill=black](21,0)circle(12pt);
\draw[fill=black](-21,0)circle(12pt);
\draw[fill=black](26.25,0)circle(12pt);
\draw[fill=black](-3,23)circle(12pt);
\draw[fill=black](-3,17)circle(12pt);
\draw[fill=black](-3,15)circle(12pt);
\draw[fill=black](-3,5)circle(12pt);
\draw[fill=black](-3,3)circle(12pt);
\draw[fill=black](-3,-23)circle(12pt);
\draw[fill=black](-3,-17)circle(12pt);
\draw[fill=black](-3,-15)circle(12pt);
\draw[fill=black](-3,-5)circle(12pt);
\draw[fill=black](-3,-3)circle(12pt);
\draw[fill=black](3,23)circle(12pt);
\draw[fill=black](3,17)circle(12pt);
\draw[fill=black](3,15)circle(12pt);
\draw[fill=black](3,5)circle(12pt);
\draw[fill=black](3,3)circle(12pt);
\draw[fill=black](3,-23)circle(12pt);
\draw[fill=black](3,-17)circle(12pt);
\draw[fill=black](3,-15)circle(12pt);
\draw[fill=black](3,-5)circle(12pt);
\draw[fill=black](3,-3)circle(12pt);
\draw[fill=black](0,23)circle(12pt);
\draw[fill=black](0,20)circle(12pt);
\draw[fill=black](0,17)circle(12pt);
\draw[fill=black](-3,20)circle(12pt);
\draw[fill=black](3,20)circle(12pt);
\draw[fill=black](0,-23)circle(12pt);
\draw[fill=black](0,-20)circle(12pt);
\draw[fill=black](0,-17)circle(12pt);
\draw[fill=black](-3,-20)circle(12pt);
\draw[fill=black](3,-20)circle(12pt);
\draw[fill=black](0,3)circle(12pt);
\draw[fill=black](0,-3)circle(12pt);
\draw[fill=black](3,0)circle(12pt);
\draw[fill=black](-3,0)circle(12pt);
\draw[fill=black](0,0)circle(12pt);
\draw[fill=black](-4.5,20)circle(12pt);
\draw[fill=black](4.5,20)circle(12pt);
\draw[fill=black](-4.5,-20)circle(12pt);
\draw[fill=black](4.5,-20)circle(12pt);
\draw[fill=black](-4.5,0)circle(12pt);
\draw[fill=black](4.5,0)circle(12pt);
\draw[fill=black](0,-28)circle(12pt);
\draw[fill=black](3,-15)circle(12pt);
\draw[fill=black](-3,-15)circle(12pt);
\draw[fill=black](3,-25)circle(12pt);
\draw[fill=black](-3,-25)circle(12pt);
\draw[fill=red](0,-30)circle(15pt);
\node at (0,-34) {$(a)$};
\end{tikzpicture}
\caption{Construction of $G_{15}$ from $T_{15}$.}
\label{jkl}
\end{figure}
\section{Notations and Preliminaries}
We use similar techniques as used in \cite{ghosh}.  We repeat some of the important notations and definitions, and for a more comprehensive discussion we refer the reader to the previous paper.
\begin{definition}\label{kkv}
Let $G$ be a plane graph and $e\in E(G)$.  If $e$ is not in a bounded $3$-face of $G$, then we call $e$ a \textbf{trivial triangular-block}.  Otherwise, we recursively construct a \textbf{triangular-block} in the following way.  Start with $H$ as a subgraph of $G$, such that $E(H)=\{e\}$.
\begin{enumerate}
\item Add the other edges of the $3$-face containing $e$ to $E(H)$.
\item Take $e'\in E(H)$ and search for a bounded $3$-face containing $e'$. Add these other edge(s) in this bounded $3$-face to $E(H)$.
\item Repeat step $2$ till we cannot find a bounded $3$-face for any edge in $E(H)$.
\end{enumerate}
 We denote the triangular-block obtained from $e$ as the starting edge, by $B(e)$.
\end{definition}
Let $G$ be a plane graph. We have the following two observations:
\begin{enumerate}
\item If $H$ is a non-trivial triangular-block and $e_1, e_2\in E(H)$, then $B(e_1)=B(e_2)=H$.
\item Any two triangular-blocks of $G$ are edge disjoint.
\end{enumerate}
Let $\mathcal{B}$ be the family of all triangular-blocks of $G$. From (2) of the observations above we have $$e(G)=\sum\limits_{B\in\mathcal{B}}e(B),$$ 
where $e(G)$ and $e(B)$ are number of edges of $G$ and $B$ respectively. For a triangular-block $B$ in a plane graph $G$, we may call $e(B)$ as the \emph{contribution of $B$ to the number of edges in $G$}.

Next we list out all possible triangular-blocks (together with their notations) that a given
$\Theta_6$-free plane graph may contain. 


We denote a trivial triangular-block as $B_2$, see Figure~\ref{hn2}. The number $2$ is to indicate that the triangular-block contains only $2$ vertices. Similarly in the notation of the other triangular-blocks, the number indicates the number of vertices the triangular-block contains.

It is obvious that there is only one $3$-vertex triangular-block, see Figure~\ref{hn2}. We denote the triangular-block by $B_3$.

A triangular-block of $4$ vertices is obtained from a triangular-block of $3$ vertices using (2) of Definition~\ref{kkv}. With this it can be checked that there are two possible triangular-blocks of $4$-vertices, see triangular-block $B_{4,}$ and triangular-block $B_{4,b}$ in Figure~\ref{hn2}.

A triangular-block of $5$ vertices is obtained from a triangular-block of $4$ vertices using (2) of Definition~\ref{kkv}. It can also be checked that we get $5$-vertex triangular-block either $B_{5,a}$ or $B_{5,b}$ or $B_{5,d}$ (see Figure~\ref{hn}) from the $4$-vertex triangular-block $B_{4,b}$. On the other hand, we get the $5$-vertex triangular-block either $B_{5,a}$ or $B_{5,c}$ (see Figure~\ref{hn}) from the triangular-block $B_{4,a}$. Therefor, there are only four $5$-vertex triangular-blocks. 

\subsubsection*{Triangular-blocks on $5$ vertices}
There are four types of blocks on $5$ vertices (See Figure \ref{hn}). Notice that $B_{5,a}$ is a $K_5^- $ .  
\begin{figure}[ht]
\centering
\begin{tikzpicture}[scale=0.13]
\draw[ultra thick](0,10)--(8.7,-5)--(-8.7,-5)--(0,10)(0,10)--(0,0)(-8.7,-5)--(0,5)--(8.7,-5)(-8.7,-5)--(0,0)--(8.7,-5);
\draw[fill=black](0,10)circle(25pt);
\draw[fill=black](0,0)circle(25pt);
\draw[fill=black](0,5)circle(25pt);
\draw[fill=black](8.7,-5)circle(25pt);
\draw[fill=black](-8.7,-5)circle(25pt);
\node at (0,-8) {$B_{5,a}$};
\end{tikzpicture}\qquad\qquad
\begin{tikzpicture}[scale=0.2]
\draw[fill=black](-4,-4)circle(13pt);
\draw[fill=black](4,-4)circle(13pt);
\draw[fill=black](4,4)circle(13pt);
\draw[fill=black](-4,4)circle(13pt);
\draw[fill=black](0,0)circle(13pt);
\draw[ultra thick](-4,-4)--(4,-4)--(4,4)--(-4,4)--(0,0)--(4,4)(0,0)--(4,-4)
(0,0)--(-4,4)(0,0)--(-4,-4)(-4,4)--(-4,-4);
\node at (0,-7){$B_{5,b}$};
\end{tikzpicture}\qquad\qquad
\begin{tikzpicture}[scale=0.23]
\draw[fill=black](2,0)circle(11pt);
\draw[fill=black](0,-4)circle(11pt);
\draw[fill=black](0,4)circle(11pt);
\draw[fill=black](4,0)circle(11pt);
\draw[fill=black](-4,0)circle(11pt);
\draw[ultra thick](-4,0)--(0,4)--(4,0)--(0,-4)--(-4,0);
\draw[ultra thick](2,0)--(0,4)(0,-4)--(0,4)(2,0)--(0,-4)(2,0)--(4,0);
\node at (0,-7){$B_{5,c}$};
\end{tikzpicture}\qquad\qquad
\begin{tikzpicture}[scale=0.1]
\draw[ultra thick](0,10)--(9.5,3.1)--(5.9,-8.1)--(-5.9,-8.1)--(-9.5,3.1)--(0,10)(0,10)--(5.9,-8.1)(0,10)--(-5.9,-8.1);
\draw[fill=black](0,10)circle(30pt);
\draw[fill=black](9.5,3.1)circle(30pt);
\draw[fill=black](-9.5,3.1)circle(30pt);
\draw[fill=black](5.9,-8.1)circle(30pt);
\draw[fill=black](-5.9,-8.1)circle(30pt);
\node at (0,-13) {$B_{5,d}$};
\end{tikzpicture}
\caption{Triangular-blocks with $5$ vertices}
\label{hn}
\end{figure}
\subsubsection*{Triangular-blocks on $2$, $3$  and $4$ vertices}
The $2$-vertex and $3$-vertex triangular-blocks are simply $K_2$ (trivial triangular-block) and $K_3$ (triangle) respectively. There are two triangular-blocks on $4$ vertices (see the last two graphs in Figure \ref{hn2}). Observe that $B_{4,a}$ is a $K_4$.
\begin{figure}[ht]
\centering
\begin{tikzpicture}[scale=0.18]
\draw[fill=black](-6,-4)circle(14pt);
\draw[fill=black](6,-4)circle(14pt);
\draw[ultra thick](6,-4)--(-6,-4);
\node at (0,-7){$B_2$};
\end{tikzpicture}\qquad\qquad 
\begin{tikzpicture}[scale=0.18]
\draw[fill=black](0,5)circle(15pt);
\draw[fill=black](-6,-4)circle(15pt);
\draw[fill=black](6,-4)circle(15pt);
\draw[ultra thick](6,-4)--(-6,-4)(-6,-4)(6,-4)--(0,5)--(-6,-4);
\node at (0,-7){$B_3$};
\end{tikzpicture}\qquad\qquad   
\begin{tikzpicture}[scale=0.21]
\draw[fill=black](0,0)circle(13pt);
\draw[fill=black](0,5)circle(13pt);
\draw[fill=black](-5,-4)circle(13pt);
\draw[fill=black](5,-4)circle(13pt);
\draw[ultra thick](0,0)--(0,5)(5,-4)--(-5,-4)(-5,-4)--(0,0)--(5,-4)--(0,5)--(-5,-4);
\node at (0,-7){$B_{4,a}$};
\end{tikzpicture}\qquad\qquad
\begin{tikzpicture}[scale=0.28]
\draw[fill=black](-3,0)circle(9pt);
\draw[fill=black](0,3)circle(9pt);
\draw[fill=black](3,0)circle(9pt);
\draw[fill=black](0,-3)circle(9pt);
\draw[ultra thick](-3,0)--(0,3)--(3,0)--(0,-3)--(-3,0)(-3,0)--(3,0);
\node at (0,-7){$B_{4,b}$};
\end{tikzpicture}
\caption{Triangular-blocks with 2, 3 and 4 vertices}
\label{hn2}
\end{figure}

\begin{lemma}
Let $G$ be a $\Theta_6$-free plane graph. Then $G$ contains no $k$-vertex triangular-block for all $k\geq 6$.
\end{lemma}
\begin{proof}
Following the recursive definition of a triangular-block, showing that a triangular-block on $6$ vertices contains a $\Theta_6$ is enough to complete proof of the lemma. Recall that a $6$-vertex triangular-block is obtained from a $5$-vertex triangular-block. Recall that there are four $5$-vertex triangular-blocks. 

We only see the case that the $5$-vertex triangular-block is $B_{5,b}$ and similar argument can be given for the remaining $5$-vertex triangular-blocks. Suppose that a $6$-vertex triangular-block $B$ is obtained from a triangular-block $B_{5,b}$. Let the outer $4$-cycle of the $B_{5,b}$ be $v_1v_2v_3v_4v_1$ and let its inner vertex be $v_5$. Denote the vertex in $B$ but not in the $B_{5,b}$ by $v_6$. From Definition~\ref{kkv} (2), $v_6$ must be adjacent to both end vertices of an edge in $\{v_1v_2, v_2v_3, v_3v_4, v_4v_1\}$ and form a $3$-face in $B$. Without loss of generality assume  $v_6v_1v_2$ be a $3$-face in $B$. In this case we get a non induced $6$-cycle $v_1v_5v_4v_3v_2v_6v_1$. Therefore $B$ contains a $\Theta_6$, which is a contradiction.   
\end{proof}

\begin{definition}
Let $G$ be a plane graph. 
 A vertex $v$ in $G$ is called a \textbf{junction vertex} if it is shared by at least two triangular-blocks of $G$. 
 \end{definition}
\begin{definition}
 Let $B$ be a triangular-block in a plane graph $G$. The \textbf{contribution of $B$ to the  number of vertices in $G$}, denoted by $\hat{v}(B)$, is defined as 
$$\hat{v}(B)=\sum\limits_{v\in V(B)}\frac{1}{\mbox{\# triangular-blocks sharing $v$}}.$$
 \end{definition}
For a plane graph $G$ and $\mathcal{B}$, the set of all triangular-blocks of $G$, one can see that $$v(G)=\sum\limits_{B\in \mathcal{B}}\hat{v}(B).$$


\begin{definition} Let $G$ be a 2-connected plane graph containing at least 2 triangular-blocks. Let $B$ be a triangular-block in $G$. An edge of $B$ is called an \textbf{exterior edge}, if it is on a boundary of non triangular face of $G$. Otherwise, we call it as an \textbf{interior edge}. A path $P$ in $B$ given by $x_1x_2\dots x_m$, where $x_1$ and $x_m$ are the only junction vertices in $P$ and $x_ix_{i+1}$, is an exterior edge for all $i\in\{1,2,3,\dots,m-1 \}$ is called an \textbf{exterior path} in $B$. A non triangular face in $G$ to which $P$ is incident with is called an \textbf{exterior face} of $B$ with the exterior path $P$.
\end{definition}

Let $G$ be a 2-connected $\Theta_6$-free plane graph containing at least two triangular-blocks. Next, we define the \emph{contribution of a triangular-block $B$ in $G$ to the number of faces of $G$}.

Let $\phi$ be a non-triangular face in $G$ with boundary cycle $x_1x_2x_3 \dots x_m x_1$, where $m\geq 4$. We may denote the cycle by $\phi$.  Consider the cycle $\phi'$ which is obtained by deleting all the vertices and edges in $G$ except vertices and edges of $\phi$. 

We construct a cycle $\phi''$, whose size is at most $m$, from $\phi'$ in the following way. For each cherry $x_ix_jx_k$ of the cycle $\phi'$, an exterior path for some triangular-block $B$ of $G$ and $x_ix_k\in E(B)$, delete $x_j$ and join $x_i$ and $x_k$ with an edge. We call such a cherry as a \textit{bad cherry}, and the cycle $\phi''$ as the \textit{refinement} of $\phi$.

For instance, consider a plane graph $G$ shown in the Figure \ref{wsde} (left). $G$ has 10 triangular-blocks; five $B_2$, one $B_3$, $B_{4,a}$, $B_{4,b}$, $B_{5,b}$ and $B_{5,c}$. Consider the non-triangular face, say $\phi$, with the boundary cycle $x_1x_2x_3x_4x_5x_6x_7x_8x_9x_{10}x_{11}x_{12}x_{13}x_1$ (see the boundary of the shaded region). Except two trivial triangular-blocks, all the remaining triangular-blocks are incident to $\phi$. Related to $\phi$, the bad cherries are $x_7x_6x_5$ and $x_9x_{10}x_{11}$. However, the cherry $x_{11}x_{12}x_{13}$ is not a bad cherry though it is an exterior path of the $B_{4,b}$; as $x_{11}x_{13}\notin E(B_{4,b})$. It can be seen that the refinement $\phi''$ of $\phi$ is  $x_1x_2x_3x_4x_5x_7x_8x_9x_{11}x_{12}x_{13}x_1$, which is of size $11$.
\begin{figure}[ht]
\centering
\begin{tikzpicture}[scale=0.22]
\filldraw[black, fill opacity=0.1] plot coordinates {(10,0) (6,0) (4.24,4.24)(7.1,7.1)(0,10)(-1.74,4)(-7.1,7.1)(-10,0)(-7.1,-7.1)(-1.74,-4)(0,-10)(2.56,-6)(7.1,-7.1)(10,0)} ;
\draw[thick](10,0)--(7.1,7.1)--(0,10)--(-7.1,7.1)--(-10,0)--(-7.1,-7.1)--(0,-10)--(7.1,-7.1)--(10,0);
\draw[thick](0,-10)--(2.56,-6)--(7.1,-7.1)--(4.64,-11.1)--(0,-10)(2.56,-6)--(3.6,-8.55);
\draw[thick](0,-10)--(-2.56,-6)--(-7.1,-7.1)--(-1.74,-4)--(0,-10)(-2.56,-6)--(-1.74,-4);
\draw[thick](10,0)--(6,0)--(4.24,4.24)--(7.1,7.1)--(6.47,2.68)--(6,0)(10,0)--(6.47,2.68)--(4.24,4.24);
\draw[thick](0,10)--(-2.56,6)--(-7.1,7.1)--(-1.74,4)--(0,10)(-2.56,6)--(-1.74,4)(0,10)--(-4.64,11.1)--(-7.1,7.1);
\draw[thick](-10,0)--(-12,-5)--(-7.1,-7.1);
\draw[fill=black](10,0)circle(10pt);
\draw[fill=black](-10,0)circle(10pt);
\draw[fill=black](0,10)circle(10pt);
\draw[fill=black](0,-10)circle(10pt);
\draw[fill=black](7.1,7.1)circle(10pt);
\draw[fill=black](-7.1,7.1)circle(10pt);
\draw[fill=black](-7.1,-7.1)circle(10pt);
\draw[fill=black](7.1,-7.1)circle(10pt);
\draw[fill=black](2.56,-6)circle(10pt);
\draw[fill=black](3.6,-8.55)circle(10pt);
\draw[fill=black](4.64,-11.1)circle(10pt);
\draw[fill=black](-2.56,-6)circle(10pt);
\draw[fill=black](-1.74,-4)circle(10pt);
\draw[fill=black](-2.56,6)circle(10pt);
\draw[fill=black](-4.64,11.1)circle(10pt);
\draw[fill=black](-1.74,4)circle(10pt);
\draw[fill=black](6,0)circle(10pt);
\draw[fill=black](4.24,4.24)circle(10pt);
\draw[fill=black](6.47,2.68)circle(10pt);
\draw[fill=black](-12,-5)circle(10pt);
\node at (0,0){$\phi$};
\node at (12,0){$x_1$};
\node at (6,-1){$x_2$};
\node at (3.5,3.5){$x_3$};
\node at (8,8){$x_4$};
\node at (1,11){$x_5$};
\node at (-2.56,3){$x_6$};
\node at (-8,8){$x_7$};
\node at (-12,0){$x_8$};
\node at (-2.56,-3){$x_{10}$};
\node at (-8,-8){$x_{9}$};
\node at (2.56,-5){$x_{12}$};
\node at (-1,-11){$x_{11}$};
\node at (8.5,-8.5){$x_{13}$};
\node at (14,-2){$\Longrightarrow$};
\end{tikzpicture}
\begin{tikzpicture}[scale=0.22]
\draw[black, fill opacity=0.15] plot coordinates {(10,0) (6,0) (4.24,4.24)(7.1,7.1)(0,10)(-7.1,7.1)(-10,0)(-7.1,-7.1)(0,-10)(2.56,-6)(7.1,-7.1)(10,0)} ;
\draw[fill=black](10,0)circle(10pt);
\draw[fill=black](6,0)circle(10pt);
\draw[fill=black](4.24,4.24)circle(10pt);
\draw[fill=black](7.1,7.1)circle(10pt);
\draw[fill=black](0,10)circle(10pt);
\draw[fill=black](-7.1,7.1)circle(10pt);
\draw[fill=black](-10,0)circle(10pt);
\draw[fill=black](-7.1,-7.1)circle(10pt);
\draw[fill=black](0,-10)circle(10pt);
\draw[fill=black](2.56,-6)circle(10pt);
\draw[fill=black](7.1,-7.1)circle(10pt);
\node at (0,0){$\phi''$};
\node at (12,0){$x_1$};
\node at (6,-1){$x_2$};
\node at (3.5,3.5){$x_3$};
\node at (8,8){$x_4$};
\node at (1,11){$x_5$};
\node at (-8,8){$x_7$};
\node at (-12,0){$x_8$};
\node at (-8,-8){$x_{9}$};
\node at (2.56,-5){$x_{12}$};
\node at (-1,-11){$x_{11}$};
\node at (8.5,-8.5){$x_{13}$};
\end{tikzpicture}
\caption{An example showing how to compute the size of the refinement of a non-triangular face in a plane graph.}
\label{wsde}
\end{figure}
\begin{definition} Let $G$ be a 2-connected $\Theta_6$-free plane graph containing at least two triangular-blocks and $\delta(G)\geq 3$. Let $B$ be a triangular-block and $\phi_1,\phi_2,\dots, \phi_m$ be all the exterior faces incident to $B$. Consider an exterior face $\phi_i$ of $B$ for some $i\in\{1,2,\dots,m\}$ with an exterior path $P$ of the triangular-block. We define the \textbf{contribution of $B$ to the face size of $\phi_i$}, denoted by $f_{\phi_i}(B)$, as follows.
\begin{enumerate}
\item If $P$ is not a bad cherry, $$f_{\phi_i}(B)=\frac{\mbox{length of $P$}}{\mbox{length of $\phi_i''$}}.$$
\item If $P$ is a bad cherry, $$f_{\phi_i}(B)=\frac{1}{\mbox{length of $\phi_i''$}}.$$
\end{enumerate}
Where $\phi_i''$ is the refinement of $\phi_i$.

The \textbf{contribution of the triangular-block $B$ to the number of faces of the graph $G$}, denoted as $\hat{f}(B)$, is defined as: 
$$\hat{f}(B)=\sum\limits_{i=1}^{m}f_{\phi_i}(B)+\left(\mbox{\# triangular faces in $B$}\right).$$
\end{definition}

Let $G$ be a 2-connected $\Theta_6$-free plane graph containing at least two triangular-blocks and $\delta(G)\geq 3$. For $\mathcal{B} $ be the family of triangular-blocks in $G$ we have,  $f(G)=\sum\limits_{B\in\mathcal{B}}\hat{f}(B)$.  

\section{Proof of Theorem \ref{vcb}.}
We begin by outlining our proof. Consider a plane drawing of $G$. Let $\mathcal{B}$ be the family of all triangular-blocks of $G$. The main target of the proof is to show that 
\begin{align}
24f(G)-17e(G)+6v(G)\leq 0\label{eq}.
\end{align}
where $v(G)$ is number of vertices of in $G$ (in this case $n$).

Once we prove (\ref{eq}), then using the Euler's Formula, $e(G)=f(G)+v(G)-2$, we can finish proof of the theorem.

To prove (\ref{eq}), we show the existence of a partition $\{\mathcal{P}_1, \mathcal{P}_2,\dots,\mathcal{P}_m\}$ of $\mathcal{B}$ such that 

$$24\sum\limits_{B\in \mathcal{P}_i}\hat{f}(B)-17\sum\limits_{B\in \mathcal{P}_i}e(B)+6\sum\limits_{B\in\mathcal{P}_i}\hat{v}(B)\leq 0,\ \mbox{for all}\  i\in\{1,2,3\dots,m\}.$$ 

Since $f(G)=\sum\limits_{B\in \mathcal{B}}\hat{f}(B)$, $v(G)=\sum\limits_{B\in\mathcal{B}}\hat{v}(B)$ and $e(G)=\sum\limits_{B\in \mathcal{B}}e(B)$ we have, 
\begin{align*}
24f(G)-17e(G)+6v(G)&=24\sum\limits_{i=1}^{m}\sum\limits_{B\in\mathcal{P}_i}\hat{f}(B)-17\sum\limits_{i=1}^{m}\sum\limits_{B\in\mathcal{P}_i}e(B)+6\sum\limits_{i=1}^{m}\sum\limits_{B\in\mathcal{P}_i}\hat{v}(B)\\&=\sum\limits_{i=1}^{m}\bigg(24\sum\limits_{B\in\mathcal{P}_i}\hat{f}(B)-17\sum\limits_{B\in\mathcal{P}_i}e(B)+6\sum\limits_{B\in\mathcal{P}_i}\hat{v}(B)\bigg)\leq 0.
\end{align*}
To verify the existence of such a partition of triangular-blocks of $G$, we prove a sequence of claims about an upper bound of $24\hat{f}(B)-17e(B)+6\hat{v}(B)$ for each type of triangular-block $B$ which may possibly exist in $G$.

We give the arguments in the proof of Claim~\ref{aau1}, Claim~\ref{km} and Claim~\ref{gm} in detail. To avoid redundancy of arguments, we might skip the detail in the proof for the remaining claims.

For simplicity of arguments, we define a function $g:\mathcal{B}\longrightarrow \mathbb{R}$ as:  $$g(B):=24\hat{f}(B)-17e(B)+6\hat{v}(B).$$

\begin{claim}\label{aau1}
Let $B$ be a $B_{5,a}$ triangular-block in $G$. Then $g(B)\leq 0$.
\end{claim}
\begin{proof}
Let the exterior vertices of $B$ be labeled as $x_1,x_2$ and $x_3$ as shown in Figure \ref{b5a}.

Observe that for any exterior vertex pair $x_i$ and $x_j$ of $B$, an $(x_i,x_j)$-path with all its interior vertices not in $B$ must be of length at least 5. Otherwise, it is easy to find a non-induced $6$-cycle using vertices of the path and the triangular-block. For simplicity we might call an $(x_i,x_j)$-path with all its interior vertices not in $B$ as  \textit{shorter path of $B$} if its length is at least 2 and at most 4.

Since the graph is $2$-connected and $n\geq 6$, $B$ contains at least $2$ junction vertices. We distinguish two cases based on the number of junction vertices of $B$.
\begin{figure}[ht]
\centering
\begin{tikzpicture}[scale=0.13]
\draw[ultra thick](0,10)--(8.7,-5)--(-8.7,-5)--(0,10)(0,10)--(0,0)(-8.7,-5)--(0,5)--(8.7,-5)(-8.7,-5)--(0,0)--(8.7,-5);
\draw[fill=black](0,10)circle(20pt);
\draw[fill=black](0,0)circle(20pt);
\draw[fill=black](0,5)circle(20pt);
\draw[fill=black](8.7,-5)circle(20pt);
\draw[fill=black](-8.7,-5)circle(20pt);
\node at (-11.5,-4.5){$x_1$};
\node at (11.5,-4.5){$x_3$};
\node at (0,12.5){$x_2$};
\end{tikzpicture}
\caption{$B_{5,a}$ triangular-block}
\label{b5a}
\end{figure}
\begin{enumerate}
\item \textbf{$B$ contains exactly two junction vertices}

We do for the case when $x_2$ and $x_3$ are the junction vertices, and similar argument can be given for other pairs too. 

Let the exterior faces of the exterior edge $x_2x_3$ and the exterior path $x_2x_1x_3$ be respectively $\phi_1$ and $\phi_2$. Since there is no shorter $(x_2,x_3)$-path, the size of $\phi_1$ (and hence $\phi_1''$) is $6$. On the other hand since $x_1$ is not a junction vertex and $B$ has no shorter $(x_2,x_3)$-path,  the size of $\phi_2$ is at least $7$. Since $x_2x_1x_3$ is a bad cherry, the refinement $\phi_2''$ has size at least 6. 
Since the number of triangular faces in $B$ is $5$, we have $\hat{f}(B)\leq 5+1/6+1/6$. $\hat{v}(B)$ is maximum if we assume each junction vertex of the triangular-block is shared by two triangular-blocks. Thus, $\hat{v}(B)\leq 3+1/2+1/2$.  From the fact that $e(B)=9$, we get $g(B)\leq -1$.
    
\item \textbf{$B$ contains three junction vertices}

Since $B$ has three junction vertices, to get a maximum upper bound of $\hat{v}(B)$ we may assume that each junction vertex is shared by two triangular-blocks and we get $\hat{v}(B)\leq 2+3/2.$ Let $\phi_1, \phi_2$ and $\phi_3$ be the exterior face of $x_1x_2, x_2x_3$ and $x_3x_1$ respectively. It can be seen that each of the face and refinement has size at least 6. Thus, $\hat{f}(B)\leq 5+3/6$. Therefore, using $e(B)=9$ we get $g(B)\leq 0.$ 

Next we give our reason why $\phi_1''$ has size at least $6$, and similar reasons can be given why $\phi_2''$ and $\phi_3''$ have size at least 6 too. 

Suppose that $\phi_1$ is a $3$-face. Let $x_1y_1x_2$ be the boundary of the face. Since $x_3$ is a junction vertex, $y_1$ and $x_3$ are different vertices, but this implies we have a shorter $(x_1,x_2)$-path, namely $x_1y_1x_2$, which is a contradiction.

Suppose that $\phi_1$ is a $4$-face and let $x_1y_1y_2x_2$ be the boundary of $\phi_1$. Since the vertex $x_3$ is a junction, $y_1$ and $y_2$ are different from $x_3$. This results the path $x_1y_1y_2x_2$ is a short $(x_1,x_2)$-path of $B$, and this is a contradiction. 

Suppose $\phi_1$ is a $5$-face and let $x_1y_1y_2y_3x_2$ be the boundary of $\phi_1$. For the same reason given above the vertices $y_1$ and $y_3$ are different from $x_3$. Since $B$ has no shorter $(x_1,x_2)$-path, we may assume $y_2$ and $x_3$ are identical vertices. However in this case, we have a shorter $(x_1,x_3)$-path, namely $x_1y_1x_3$, and this is a contradiction. \qedhere
\end{enumerate}
\end{proof}
\begin{claim}\label{km}
Let $B$ be a $B_{5,b}$ triangular-block in $G$. Then $g(B)\leq 0$.
\end{claim}
\begin{proof}
Denote the four exterior vertices the triangular-block by $x_1,x_2,x_3, x_4$ and the interior vertex of the triangular-block by $x_5$ as shown in Figure~\ref{cdc}.

Observe that for any exterior vertex pair $x_i$ and $x_j$ of $B$, an $(x_i,x_j)$-path with all its interior vertices not in $B$ must be of length at least 5. Otherwise, it is easy to find a non-induced $6$-cycle using vertices of the path and the triangular-block. For simplicity we might call an $(x_i,x_j)$-path with all its interior vertices not in $B$ as  \textit{shorter path of $B$} if its length is at least 2 and at most 4.   
\begin{figure}[ht]
\centering
\begin{tikzpicture}[scale=0.2]
\draw[fill=black](-4,-4)circle(13pt);
\draw[fill=black](4,-4)circle(13pt);
\draw[fill=black](4,4)circle(13pt);
\draw[fill=black](-4,4)circle(13pt);
\draw[fill=black](0,0)circle(13pt);
\draw[ultra thick](-4,-4)--(4,-4)--(4,4)--(-4,4)--(0,0)--(4,4)(0,0)--(4,-4)
(0,0)--(-4,4)(0,0)--(-4,-4)(-4,4)--(-4,-4);
\node at (-4,-5.5){$x_1$};
\node at (-4,5.5){$x_2$};
\node at (4,5.5){$x_3$};
\node at (4,-5.5){$x_4$};
\node at (0,-2){$x_5$};
\end{tikzpicture}
\caption{$B_{5,b}$ triangular-block}
\label{cdc}
\end{figure}

Since $G$ contains no cut vertex, the number of junction vertices of $B$ is at least $2$. Next we distinguish three subcases to address an upper bound of $g(B)$.
\begin{enumerate}
\item \textbf{$B$ has only two junction vertices}

In this case $\hat{v}(B)\leq 3+1/2+1/2.$ 
Without loss of generality we may assume that the the junction vertices are either $x_1$ and $x_2$ or $x_1$ and $x_3$. 

Let the junction vertices be $x_1$ and $x_2$. Since $B$ has no shorter $(x_1,x_2)$-path, the exterior faces of the exterior paths $x_1x_2$ and $x_1x_4x_4x_2$ are at least $6$ and $8$ respectively. Even more, the length the refinement of the two faces are at least $6$ and $8$ respectively. Hence $\hat{f}(B)\leq 4+1/6+3/8$. Using $e(B)=8$, we get $g(B)\leq -3.$ 

On the other hand if the junction vertices are $x_1$ and $x_3$, then the refinements of the exterior faces of $x_1x_2x_3$ and $x_2x_4x_3$ have size at least $6$. Hence, $\hat{f}(B)\leq 4+4/6.$  From the fact that $e(B)=8,$ we get $g(B)\leq 0.$  

\item \textbf{$B$ has only three junction vertices}

Since $B$ has three junction vertices, $\hat{v}(B)\leq 2+1/2+1/2+1/2.$  Without loss of generality assume that the junction vertices are $x_1, x_2$ and $x_3$. Denote the exterior face of $x_1x_2, x_2x_3$ and $x_1x_4x_3$ be respectively $\phi_1, \phi_2$ and $\phi_3$ respectively. 

It can be seen that size of $\phi_1$ $\phi_2$, and $\phi_3$ (and hence their refinement) are at least 6. Thus $\hat{f}(B)\leq 4+4/6.$, using $e(B)=8$, we have $g(B)\leq -3. $

Next we give the reason why the exterior faces have size of at least 6. Let the size of $\phi_1$ be 3 and  and is $x_1y_1x_2$. It can be checked that $y_1$ can never be $x_3$ and hence $B$ contains a shorter $(x_1,x_2)$-path, which is a contradiction. 

Let the size of $\phi_1$ be $4$ and the face is $x_1y_1y_2x_2$. Since $x_2$ is a junction vertex, $y_2$ can not be $x_3$. If $y_1$ is different from $x_3$, then $B$ contains a shorter $(x_1,x_2)$-paths, namely $x_1y_1y_2x_2$, and hence a contradiction. If $y_1$ and $x_3$ identical, we still get a shorter $(x_2,x_3)$-paths, namely $x_2y_2x_3$, which is a contradiction again. 

Let the size of $\phi_1$ be $5$ and the face is $x_1y_1y_2y_3x_2$. Again for the reason that $x_2$ is a junction vertex, $y_3$ is different from $x_3$. If both $y_1$ and $y_2$ are different from $x_3$, then the path $x_1y_1y_2y_3x_2$ is a shorter $(x_1,x_2)$-path and hence a contradiction. So we assume when $y_1$ and $x_3$ are the same vertices or $y_2$ and $x_3$ are the same vertices. In the former case,  we get a shorter $(x_2,x_3)$-path, namely $x_3y_2y_3x_2$, which is a contradiction. In the later case, we get a shorter $(x_1,x_3)$-path, namely $x_1y_1x_3$, and this results a contradiction. 

Now we show why the size of $\phi_3$ is at least 6.
Let the size of $\phi_3$ be 4 and is $x_1y_1x_3x_4$. It can be checked that $y_1$ is different from $x_2$. Thus we have got a shorter $(x_1,x_3)$-path, namely $x_1y_1x_3$, which is a contradiction. Similarly, if we assume that the size of $\phi_3$ is 5 and the boundary cycle is $x_1y_1y_2x_3x_4$, then it can be shown that the path $x_1y_1y_2x_3$ is a shorter $(x_1,x_3)$-path

\item \textbf{$B$ has four junction vertices}

Since $B$ has four junction vertices, $\hat{v}(B)\leq 1+1/2+1/2+1/2+1/2.$  Let the exterior faces of the exterior edges $x_1x_2, x_2x_3, x_3x_4$ and $x_4x_1$ be $\phi_1, \phi_2, \phi_3$ and $\phi_4$ respectively. We claim that each face has size at least $6$. This results the refinement of each face has size at least $6$ and $\hat{f}(B)\leq 4+1/6+1/6+1/6+1/6$. Hence using $e(B)=8$, we get $g(B)\leq -6.$

Next we give our proof why each exterior face of the triangular-block has size at least $6$. We show for the case of $\phi_1$, and similar argument can be given for the rest. 

Suppose $\phi_1$ be a $3$-face and let its boundary be $x_1y_1x_2$. Since $x_3$ and $x_4$ are junction vertices, $y_1$ can not $x_3$ or $x_4$ vertex. On the other hand, if $y_1$ is different from $x_3$ and $x_4$, then we have a shorter $(x_1,x_2)$-path, which is a contradiction. 

Let $\phi_1$ be a $4$-face and its boundary be $x_1y_1y_2x_2$. Since $x_2$ is a junction vertex, $y_2$ can not be $x_3$ and since $x_1$ is a junction vertex, $y_1$ can not be $x_4$. If both $y_1$ and $y_2$ are different from $x_3$ and $x_4$ respectively, then we have a shorter $(x_1,x_2)$-path, which is a contradiction. Thus we may assume that $y_1$ and $x_3$ are identical vertices. However in this case we get a shorter $(x_2,x_3)$-path, namely $x_2y_2x_3$, which is a contradiction. Similarly, if $y_2$ and $x_4$ are identical vertices we also have a shorter $(x_1,x_4)$-path, namely $x_1y_1x_4$, which is again a contradiction. 

Lastly, suppose $\phi_1$ be a $5$-face and its boundary be $x_1y_1y_2y_3x_2$. Since $x_1$ is a junction vertex, $y_1$ and $x_4$ can not be identical. In the same way $x_3$ and $y_3$ can not be identical vertices. Hence we way assume the case that $y_2$ is identical with $x_4$ or $y_3$ is identical with $x_4$. In the former case we get a shorter $(x_1,x_4)$-path namely $x_1y_1x_4$, which is a contradiction. On the other hand in the later case, we get a shorter $(x_1,x_4)$-path namely $x_1y_1y_2x_4$, which is again a contradiction. \qedhere 
\end{enumerate}
\end{proof}
\begin{claim}\label{gm}
Let $B$ be a $B_{5,c}$ triangular-block in $G$. Then $g(B)\leq 0$.
\end{claim}
\begin{proof}
Let the exterior vertices of the triangular-block be $x_1,x_2,x_3$ and $x_4$ as shown in Figure \ref{dd}.  
\begin{figure}[ht]
\centering
\begin{tikzpicture}[scale=0.3]
\draw[fill=black](2,0)circle(9pt);
\draw[fill=black](0,-4)circle(9pt);
\draw[fill=black](0,4)circle(9pt);
\draw[fill=black](4,0)circle(9pt);
\draw[fill=black](-4,0)circle(9pt);
\draw[ultra thick](-4,0)--(0,4)--(4,0)--(0,-4)--(-4,0);
\draw[ultra thick](2,0)--(0,4)(0,-4)--(0,4)(2,0)--(0,-4)(2,0)--(4,0);
\node at (-5,0){$x_1$};
\node at (5,0){$x_3$};
\node at (0,5){$x_2$};
\node at (0,-5){$x_4$};
\node at (0.8,0){$x_5$};
\end{tikzpicture}
\caption{$B_{5,c}$ triangular-block}
\label{dd}
\end{figure}
Since $\delta(G)\geq 3$, $x_1$ is a junction vertex. However, $x_3$ may or may not be a junction vertex. Next, we distinguish two cases.

\begin{enumerate}
\item \textbf{$x_3$ is a junction vertex.} 

Observe that the set of all possible exterior paths of $B$ is $\{x_1x_4, x_1x_2, x_2x_3, x_3x_4, \\ x_1x_2x_3, x_1x_4x_3\}.$ It can be shown that each possible exterior face of an exterior path has size at least 6; later we show the case when $x_1x_4$ is an exterior path of $B$, and similar argument can be given for the other possibilities.  Thus, $\hat{f}(B)\leq 4+4/6$. Since, $\hat{v}(B)\leq 3+1/2+1/2$ and  $e(B)=8$, we have $g(B)\leq 0.$

Let $\phi_1$ be an exterior face of the exterior path $x_1x_4$. Notice that $x_4$ is a junction vertex.  Suppose $\phi_1$ be a $3$-face and the boundary be $x_1y_1x_4.$ As $x_1$ and $x_4$ are junction vertices, $y_1$ can not be identical with $x_2$ and $x_3$. This results a non-induced $6$-cycle, which is a contradiction.

Suppose that $\phi_1$ be a $4$-face and the boundary be $x_1y_1y_2x_4.$ Observe that if $y_1$ and $y_2$ are vertices not in $B$, then $G$ contains a non-induced $6$-cycle, which is a contradiction. So, we may assume that one of the vertices is in $B $. From the fact that $x_1$ is a junction vertex, $y_1$ and $x_2$ are different vertices. Again since $x_4$ is a junction vertex, $y_2$ and $x_3$ are different vertices. If $y_1$ and $x_3$ are identical vertices (similarly if $y_2$ and $x_2$ are identical vertices), clearly the vertex $y_2$ is not different from $x_2$ and we get a shorter $(x_3,x_4)$-path , namely $x_3y_2x_4$, and this clearly results a non-induced $6$-cycle.

Suppose that $\phi_1$ be a $5$-face and the its boundary be $x_1y_1y_2y_3x_4$. Since $G$ is $\Theta_6$-free, the at least one of the vertex in $\{y_1,y_2,y_3\}$ is in $B$.  For the reasons given above $y_1$ is different from $x_2$, and $y_3$ is different from $x_3$. If $y_1$ and $x_3$ (similarly $y_3$ and $x_2$) are identical vertices, then we get a shorter $(x_3,x_4)$-path, namely $x_3y_2y_3x_4$, which is a contradiction. On the other hand if $y_2$ and $x_3$ are identical vertices, again we get a shorter $(x_3,x_4)$-path, namely $x_3y_3x_4$, which is a contradiction. 

\item \textbf{$x_3$ is  not a junction vertex.} 

In this case either both $x_2$ and $x_4$ are junction vertices or only one of them is a junction vertex. Otherwise, the graph contains a cut vertex. So, we have the following two subcases.
\begin{enumerate}  
\item[2.1.] \textbf{Only one of the two vertices, $x_2$ or $x_4$, is a junction vertex} 

Without loss of generality, assume $x_4$ is a junction vertex. Obviously, the exterior face of the exterior edge $x_1x_4$ is of size at least $6$. Moreover, the size of the exterior face of the exterior path $x_1x_2x_3x_4$ is at least $8$. Thus, $\hat{f}(B)\leq 4+1/6+3/8$. Assuming that the junction vertices are shared by two triangular-blocks we have, $\hat{v}(B)\leq 3+1/2+1/2.$ Therefore, we get $g(B)\leq -3.$
\item[2.2.] \textbf{Both $x_2$ and $x_4$ are junction vertices}

In this case, the exterior path $x_2x_3x_4$ has an exterior face of size at least $4$ and since the exterior path $x_2x_3x_4$ is a bad cherry, the refinement has size at least 3. For similar arguments given above, it can be seen that the exterior faces of the exterior edges $x_1x_2$ and $x_1x_4$ have size at least $6$. Thus $\hat{f}(B)\leq 4+2/6+1/3$. Since the junction vertices $x_1,x_2$ and $x_4$ are shared with at least two blocks we get $\hat{v}(B)\leq 2+3/2$. Therefore, $g(B)\leq -3.$  \qedhere
\end{enumerate}
\end{enumerate}
\end{proof}
\begin{claim}
Let $B$ be a $B_{5,d}$ triangular-block in $G$. Then $g(B)\leq 0$.
\end{claim}
\begin{proof}
Let $B$ be with its vertices labeled $x_1, x_2, x_3, x_4$ and $x_5$ as seen in Figure \ref{po}.
\begin{figure}[ht]
\centering
\begin{tikzpicture}[scale=0.1]
\draw[ultra thick](0,10)--(9.5,3.1)--(5.9,-8.1)--(-5.9,-8.1)--(-9.5,3.1)--(0,10)(0,10)--(5.9,-8.1)(0,10)--(-5.9,-8.1);
\draw[fill=black](0,10)circle(25pt);
\draw[fill=black](9.5,3.1)circle(25pt);
\draw[fill=black](-9.5,3.1)circle(25pt);
\draw[fill=black](5.9,-8.1)circle(25pt);
\draw[fill=black](-5.9,-8.1)circle(25pt);
\node at (0,13){$x_3$};
\node at (13,3.1){$x_4$};
\node at (-13,3.1){$x_2$};
\node at (10,-8.1){$x_5$};
\node at (-10,-8.1){$x_1$};
\end{tikzpicture}
\caption{$B_{5,d}$ triangular-block }
\label{po}
\end{figure}
Notice that $x_2$ and $x_4$ are junction vertices. It can be checked that each of the exterior edges of the triangular-block has an exterior face whose refinement is of size at least $6$. Thus, $\hat{f}(B)\leq 3+5/6$. We estimate that $\hat{v}(B)\leq 3+1/2+1/2$. Therefore using $e(B)=7$, we get $g(B)\leq -3.$
\end{proof}
\begin{claim}
Let $B$ be a $B_{4,a}$ triangular-block in $G$. Then $g(B)\leq 0$.
\end{claim}
\begin{proof}
Consider $B$ with the exterior vertices labeled $x_1,x_2$ and $x_3$ as shown in Figure \ref{fg}(left).
Since the graph does not contain a cut vertex, the block contains at least two junction vertices. So we distinguish two cases. 
\begin{figure}[ht]
\centering
\begin{tikzpicture}[scale=0.21]
\draw[fill=black](0,0)circle(12pt);
\draw[fill=black](0,5)circle(12pt);
\draw[fill=black](-5,-4)circle(12pt);
\draw[fill=black](5,-4)circle(12pt);
\draw[ultra thick](0,0)--(0,5)(5,-4)--(-5,-4)(-5,-4)--(0,0)--(5,-4)--(0,5)--(-5,-4);
\node at (0,6.5){$x_3$};
\node at (-6.5,-4.5){$x_1$};
\node at (6.5,-4.5){$x_2$};
\node at (0,-2){$x_4$};
\end{tikzpicture}\qquad\qquad
\begin{tikzpicture}[scale=0.21]
\draw[thick,red](-5,-4)--(0,10)--(5,-4);
\draw[fill=black](0,0)circle(12pt);
\draw[fill=black](0,5)circle(12pt);
\draw[fill=black](-5,-4)circle(12pt);
\draw[fill=black](5,-4)circle(12pt);
\draw[fill=black](0,10)circle(12pt);
\draw[ultra thick](0,0)--(0,5)(5,-4)--(-5,-4)(-5,-4)--(0,0)--(5,-4)--(0,5)--(-5,-4);
\node at (0,6.5){$x_3$};
\node at (-6.5,-4.5){$x_1$};
\node at (6.5,-4.5){$x_2$};
\node at (0,-2){$x_4$};
\end{tikzpicture}
\caption{$B_{4,a}$ triangular-block}
\label{fg}
\end{figure}

\begin{enumerate}
\item \textbf{All the three vertices are junction vertices} 

In this case each of the three exterior edges has an exterior face whose refinement is of size at least $6$. Thus, $\hat{f}(B)\leq 3+3/6$. Moreover, $\hat{v}(B)\leq 1+3/2$ and $e(B)=6$. Therefore, $g(B)\leq -3.$

\item \textbf{Only two of the vertices are junction vertices.}

Without loss of generality, assume that the junction vertices are $x_1$ and $x_2$. In this case, the exterior face of the exterior edge $x_1x_2$ is of size at least $6$. However the exterior path $x_1x_3x_2$ has size at least $4$ (see Figure~\ref{fg}(right)), and notice that $x_1x_3x_2$ is a bad cherry. So, $\hat{f}(B)\leq 3+1/6+1/3$. Since, the block has two junction vertices, then $\hat{v}(B)\leq 2+1/2+1/2$. Using $e(B)=6$, we get $g(B)\leq 0$. \qedhere
\end{enumerate}
\end{proof}

\begin{claim}\label{hm}
Let $B$ be a $B_{4,b}$ triangular-block in $G$. Then $g(B)\leq 3$.
\end{claim}
\begin{proof}
Consider the triangular-block $B$ with the labeled vertices $x_1, x_2, x_3$ and $x_4$ as seen in Figure \ref{nm}(left). Clearly $x_2$ and $x_4$ are junction vertices. We distinguish two cases.
\begin{figure}[ht]
\centering
\begin{tikzpicture}[scale=0.28]
\draw[fill=black](-3,0)circle(8pt);
\draw[fill=black](0,3)circle(8pt);
\draw[fill=black](3,0)circle(8pt);
\draw[fill=black](0,-3)circle(8pt);
\draw[ultra thick](-3,0)--(0,3)--(3,0)--(0,-3)--(-3,0)(-3,0)--(3,0);
\node at (-4.2,0){$x_1$};
\node at (0,4){$x_2$};
\node at (4.2,0){$x_3$};
\node at (0,-4){$x_4$};
\end{tikzpicture}\qquad\qquad
\begin{tikzpicture}[scale=0.28]
\draw[thick, red](0,3)--(6,0)--(0,-3)--(-6,0)--(0,3);
\draw[fill=black](-3,0)circle(8pt);
\draw[fill=black](0,3)circle(8pt);
\draw[fill=black](3,0)circle(8pt);
\draw[fill=black](0,-3)circle(8pt);
\draw[fill=black](6,0)circle(8pt);
\draw[fill=black](-6,0)circle(8pt);
\draw[ultra thick](-3,0)--(0,3)--(3,0)--(0,-3)--(-3,0)(-3,0)--(3,0);
\node at (-4.2,0){$x_1$};
\node at (0,4){$x_2$};
\node at (4.2,0){$x_3$};
\node at (0,-4){$x_4$};
\node at (-7.5,0){$x_5$};
\node at (7.5,0){$x_6$};
\end{tikzpicture}
\caption{$B_{4,b}$ triangular-block}
\label{nm}
\end{figure}
\begin{enumerate}
\item \textbf{At least one vertex in $\{x_1,x_3\}$ is a junction vertex} 

Without loss of generality, assume $x_1$ is a junction vertex. In this case the sizes of the exterior faces of the exterior edges $x_1x_2$ and $x_1x_4$ are at least $6$. To get maximum upper bound of $\hat{v}(B)$ and $\hat{f}(B)$ we may assume that $x_3$ is  not a junction vertex and size of the exterior face of the exterior path $x_2x_3x_4$ is $4$. Hence, $\hat{f}(B)\leq 2+2/6+2/4$ and $\hat{v}(B)\leq 1+3/2$. Therefore using $e(B)=5$ we get $g(B)\leq -2.$
\item \label{m1}\textbf{Both $x_1$ and $x_3$ are not junction vertices}

In this case the exterior paths $x_2x_3x_4$ and $x_2x_1x_4$ have an exterior face of sizes either $4$ or at least $6$. So, we have the following subcases.
\begin{enumerate}
\item[2.1.] \textbf{Both are with size at least $6$}

Here we estimate, $\hat{f}(B)\leq 2+4/6$, $\hat{v}(B)\leq 2+1/2+1/2$. Hence, $g(B)\leq -3.$

\item[2.2.] \textbf{Only one of the exterior paths has an exterior face of size $4$}

Without loss of generality assume $x_2x_1x_4$ has an exterior face of size $4$. Let the exterior face be with boundary $x_2x_1x_4x_5x_2$.  Notice that $x_2x_5$ and $x_4x_5$ are trivial triangular-blocks. Indeed if an edge $x_2x_5$ is not a trivial triangular-block, then the edge is incident to a $3$-face with boundary $x_2y_1x_5$, where $y_1$ is a vertex not in $B$. But this results a non-induced $6$-cycle $x_1x_4x_5y_1x_2x_3x_1$, which is a contradiction as $G$ is a $\Theta_6$-free graph. Hence, considering that there is no degree 2 vertex in $G$, either $x_2$ or $x_4$ is shared by at least three triangular-blocks. Thus, $\hat{v}(B)\leq 2+1/2+1/3$. Moreover considering the size of the exterior faces of the exterior paths $x_2x_3x_4$ and $x_2x_1x_4$ we have  $\hat{f}(B)\leq 2+2/4+2/6$. Therefore, in this case we get $g(B)\leq 0$.

\item[2.3.] \label{m222} \textbf{Both exterior paths have exterior faces of size $4$}

In this case, $\hat{f}(B)\leq 2+2/4+2/4$. Let the exterior faces be with boundary $x_2x_1x_4x_5x_2$ and $x_2x_3x_4x_6x_2$ as shown in the Figure \ref{nm}(right). The vertices $x_5$ and $x_6$ can not be identical, otherwise from the assumption that $G$ contains at least $6$ vertices, being that $x_5$ and $x_6$ are identical results a multiple edge, which is a contradiction as $G$ is a simple graph. For the reason given above in Case~(2.2 ) the edges $x_2x_5, x_4x_6, x_2x_6$ and $x_4x_6$ are trivial triangular-blocks. Thus, $x_2$ and $x_4$ are junction vertices which are shared with at least $3$ triangular-blocks. That means, $\hat{v}(B)\leq 2+2/3$. Therefore we get $g(B)\leq 3.$   \qedhere
\end{enumerate}
\end{enumerate}
\end{proof}
\begin{claim}
Let $B$ be a $B_{3}$ triangular-block in $G$. Then $g(B)\leq 0$.
\end{claim}
\begin{proof}
Let the triangular-block $B$ be with vertices $x_1, x_2$ and $x_3$ as shown in Figure \ref{fhg}(left).
\begin{figure}[ht]
\centering
\begin{tikzpicture}[scale=0.18]
\draw[fill=black](0,5)circle(13pt);
\draw[fill=black](-6,-4)circle(13pt);
\draw[fill=black](6,-4)circle(13pt);
\draw[ultra thick](6,-4)--(-6,-4)(-6,-4)(6,-4)--(0,5)--(-6,-4);
\node at (0,7){$x_2$};
\node at (-8,-4.5){$x_1$};
\node at (8,-4.5){$x_3$};
\end{tikzpicture}
\caption{$B_3$ triangular-block}
\label{fhg}
\end{figure}
Due to the degree condition of $G$, each of the three vertices is a junction vertex, and hence $\hat{v}(B)\leq 1/2+1/2+1/2.$ Moreover, each exterior edges in $\{x_1x_2, x_2x_3, x_3x_1\}$ has an exterior face of size at least $4$ resulting $\hat{f}(B)\leq 1+1/4+1/4+1/4$. Therefore, using $e(B)=3$, we get $g(B)\leq 0.$





\end{proof}
\begin{claim}\label{bd}
Let $B$ be a $B_{2}$ triangular-block in $G$. Then $g(B)\leq 0$.
\end{claim}
\begin{proof}
Let the end vertices of the triangular-block be $x_1$ and $x_2$ as shown in Figure \ref{tr}(left). The two exterior faces of the triangular-block have size at least $4$. We consider two cases.
\begin{figure}[ht]
\centering
\begin{tikzpicture}[scale=0.18]
\draw[fill=black](-6,-4)circle(13pt);
\draw[fill=black](6,-4)circle(13pt);
\draw[ultra thick](6,-4)--(-6,-4);
\node at (-8,-4){$x_1$};
\node at (8,-4){$x_2$};
\end{tikzpicture}\qquad\qquad 
\begin{tikzpicture}[scale=0.18]
\draw[thick, red](-6,-4)--(-6,-8)--(6,-8)--(6,-4);
\draw[fill=black](-6,-4)circle(13pt);
\draw[fill=black](6,-4)circle(13pt);
\draw[fill=black](6,-8)circle(13pt);
\draw[fill=black](-6,-8)circle(13pt);
\draw[ultra thick](6,-4)--(-6,-4);
\node at (-8,-4){$x_1$};
\node at (8,-4){$x_2$};
\node at (-8,-8){$x_4$};
\node at (8,-8){$x_3$};
\end{tikzpicture}
\caption{$B_2$ triangular-block}
\label{tr}
\end{figure}
\begin{enumerate}
\item \textbf{Both faces have size at least $5$}

In this case $\hat{f}(B)\leq 1/5+1/5$, $\hat{v}(B)\leq 1/2+1/2$. Using $e(B)=1$, we get $g(B)\leq -7/5.$
\item \label{m3} \textbf{One of the two exterior faces is of size $4$}

Let the exterior face of size $4$, call it $\phi_1$, be with boundary $4$-cycle $x_1x_2x_3x_4x_1$ as shown in Figure~\ref{tr}(right). Notice that the other exterior face of the trivial triangular-block, call it $\phi_2$, has size at least $5$. Otherwise, it is easy to check that there is a $\Theta_6$-graph in $G$.  It is clear that the $2$-paths $x_1x_4x_3$ or $x_4x_3x_2$ (but not both) can be a bad cherry such that the refinement of $\phi_1$ is 3. We distinguish the following cases.
\begin{enumerate}
\item[2.1] \textbf{The refinement of $\phi_1$ is $4$}

Thus, $\hat{f}(B)\leq 1/4+1/5$. In this case either $x_1x_4$ or $x_2x_3$ is a trivial triangular-block. Otherwise, it is easy to show that $G$ contains $\Theta_6$. Thus either $x_1$ or $x_2$ is shared with at least $3$ triangular-blocks considering that $\delta(G)\geq 3$. Therefore, $\hat{v}(B)\leq 1/2+1/3$, and $g(B)\leq -6/5.$
\item [2.2] \textbf{The refinement of $\phi_1$ is $3$}

Without loss of generality assume that the $2$-path $x_1x_4x_3$ is a bad cherry. That means, the path is in a fixed triangular-block, say $B^*$ such $x_1x_3\in E(B^*)$. Notice that $B^*$ can be a $4$-vertex triangular-block like $B_{4,a}$ or $B_{5,c}$. But $B^*$ can not be $B_{5,a}$. Otherwise, it is easy to show $G$ contains a $\Theta_6$-graph. Moreover observe that the edge $x_2x_3$ is a trivial block and from the minimum degree condition of $G$, the vertex $x_2$ is shared by at least $3$ triangular-blocks.

If $x_1$ is shared by at least $3$ triangular-blocks, then $\hat{v}(B)\leq 1/3+1/3$. Since $\hat{f}(B)\leq 1/3+1/5$ and $e(B)=1$, we get $g(B)\leq -1/2$.

Now assume $x_1$ is shared by only two triangular-blocks, namely $B^*$ and the trivial triangular-block $x_1x_2$. Thus, $\hat{v}(B)\leq 1/2+1/3$. Moreover it is easy to check that the exterior face $\phi_2$ is with size at least $6$. Hence, $\hat{f}(B)\leq 1/3+1/6$. Therefore, in this scenario we get $g(B)\leq 0$.\qedhere  
\end{enumerate}
\end{enumerate}
\end{proof}
We notice that there is only one possible case for which a triangular-block $B$ in $G$ may assume positive $g(B)$. The only possibility is briefly explained in the proof of Claim \ref{hm}~(2.3). The triangular-block and structures of the boundary of its exterior faces are shown in Figure~\ref{nm}(right). 

Observe that the exterior faces of the exterior paths $x_2x_1x_4$ and $x_2x_3x_4$ have size $4$ and the edges $x_4x_6$, $x_4x_5$, $x_2x_6$ and $x_2x_5$ are trivial triangular-blocks. Denote the trivial triangular-blocks respectively as $B^1, B^2, B^3$, and $B^4$. Let us call such trivial triangular-blocks as \say{\textit{good blocks}}.

It is easy to show that if any of these good blocks plays a similar role with a different $B_{4,b}$ triangular-block, the graph contains a $\Theta_6$. Indeed, Suppose $B^1$ plays such a role. Thus, there is a $B_{4,b}$ triangular-block, say $B'$, such that $B^1$ is one of the four good blocks of $B'$ (see the two possible structures in Figure \ref{vb}). But in both scenario it is easy to show $G$ contains $\Theta_6 $, which is a contradiction. 
\begin{figure}[ht]
\centering
\begin{tikzpicture}[scale=0.3]
\draw[ultra thick](-3,0)--(0,3)--(3,0)--(0,-3)--(-3,0)(-3,0)--(3,0);
\draw[thick,red](0,3)--(6,0)(0,-3)--(-6,0)--(0,3);
\draw[ultra thick](-3,-6)--(3,-6)(0,-3)--(-3,-6)--(0,-9)--(3,-6)--(0,-3);
\draw[thick, red](0,-9)..controls (10,-11) and (7,0) .. (6,0);
\draw[thick,blue](0,-3)--(6,0);
\draw[fill=black](-3,0)circle(8pt);
\draw[fill=black](0,3)circle(8pt);
\draw[fill=black](3,0)circle(8pt);
\draw[fill=black](0,-3)circle(8pt);
\draw[fill=black](6,0)circle(8pt);
\draw[fill=black](-6,0)circle(8pt);
\draw[fill=black](-3,-6)circle(8pt);
\draw[fill=black](3,-6)circle(8pt);
\draw[fill=black](0,-9)circle(8pt);
\node at (0,1){$B$};
\node at (0,-5){$B'$};
\node at (-4,0){$x_1$};
\node at (0,4){$x_2$};
\node at (4,0){$x_3$};
\node at (-1.5,-3.5){$x_4$};
\node at (-7.5,0){$x_5$};
\node at (7.5,0){$x_6$};
\end{tikzpicture}
\begin{tikzpicture}[scale=0.3]
\draw[ultra thick](-3,0)--(0,3)--(3,0)--(0,-3)--(-3,0)(-3,0)--(3,0);
\draw[thick,red](0,3)--(6,0)(0,-3)--(-6,0)--(0,3);
\draw[ultra thick](6,0)--(3,-3)--(6,-6)--(9,-3)--(6,0)(3,-3)--(9,-3);
\draw[thick,red](0,-3)--(6,-6);
\draw[thick,blue](0,-3)--(6,0);
\draw[fill=black](-3,0)circle(8pt);
\draw[fill=black](0,3)circle(8pt);
\draw[fill=black](3,0)circle(8pt);
\draw[fill=black](0,-3)circle(8pt);
\draw[fill=black](6,0)circle(8pt);
\draw[fill=black](-6,0)circle(8pt);
\draw[fill=black](3,-3)circle(8pt);
\draw[fill=black](9,-3)circle(8pt);
\draw[fill=black](6,-6)circle(8pt);
\node at (0,1){$B$};
\node at (6,-4){$B'$};
\node at (-4,0){$x_1$};
\node at (0,4){$x_2$};
\node at (4,0){$x_3$};
\node at (0,-4){$x_4$};
\node at (-7.5,0){$x_5$};
\node at (7.5,0){$x_6$};
\end{tikzpicture}
\caption{Two possible structures showing two different $B_{4,b}$ triangular-blocks sharing a good block.}
\label{vb}
\end{figure}

This implies that for each $B_{4,b}$ triangular-block meeting the conditions in Claim~\ref{hm}~(2.3), correspondingly we have unique four good blocks on the boundaries of the exterior faces of the triangular-block. In particular for $B$, the corresponding four good blocks are $B^1,B^2,B^3$ and $B^4$.

Observe that each good block is with exterior face whose refinements are 4 and at least 5. Moreover, the at least one end vertex of a good block is shared by at least 3 triangular-blocks. Thus for $i\in\{1,2,3,4\}$, we have $\hat{f}(B^i)\leq 1/4+1/5$ and $\hat{v}(B^i)\leq 1/2+1/3$, which implies $g(B^i)\leq-6/5$.  

Define  $\mathcal{P}=\{B^1,B^2,B^3,B^4,B\}$. From the prove of Claim~\ref{hm}~(2.3), $g(B)\leq 3$.  Clearly,
$$\sum\limits_{B^*\in\mathcal{P}}g(B^*)\leq 3+4\left(-6/5\right)=-9/5.$$

Let $\mathcal{B}$ be the family of all triangular-blocks of $G$ and $\mathcal{B}'=\{B_1,B_2,\dots, B_k\}$ be the set of $B_{4,b}$ triangular-blocks in $G$ meeting the conditions stated in Claim~\ref{hm}~(2.3).

Define $\mathcal{P}_{i}=\{B_i^1,B_i^2,B_i^3,B_i^4,B_i\}$ for all $i\in\{1,2,3,\dots ,k\}$ where $B_i^j$, $j\in\{1,2,3,4\}$ are the corresponding four good blocks of $B_i$. 

Define $$\mathcal{P}_{k+1}=\mathcal{B}\backslash \bigcup\limits_{i=1}^k\mathcal{P}_i.$$ 
Since $g(B^*)\leq 0$ for all $B^*\in\mathcal{P}_{k+1}$, $\sum\limits_{B^*\in\mathcal{P}_{k+1}}g(B^*)\leq 0$. Let $m=k+1$. Thus we have the partition $\mathcal{P}_1,\mathcal{P}_2,\dots,\mathcal{P}_{m}$ of  $\mathcal{B}$ such that $\sum\limits_{B^*\in \mathcal{P}_i}g(B^*)\leq 0$ for all $i\in \{1,2,3,\dots,m\}$.
Therefore, 
$$24f(G)-17e(G)+6v(G)=\sum\limits_{i=1}^{m}\sum\limits_{B^*\in\mathcal{P}_i}g(B^*)\leq 0.$$
This completes the proof of Theorem \ref{vcb}.
\section{Proof of Theorem \ref{bvc}}
We show the proof for the connected graphs only. Indeed, if the graph is not connected, then we can add an edge between components by keeping the graph connected and $\Theta_6$-free. So, if we show that the theorem holds for a connected graph, then it holds for disconnected too. 

To finish the proof of the Theorem \ref{bvc}, we need the following lemma.
\begin{lemma}\label{de}
Let $G$ be an $n$-vertex ($n\geq 2)$ $ \Theta_6$-free plane graph, then $e(G)\leq\frac{18}{7}n-\frac{27}{7}$.
\end{lemma}
\begin{proof}
First, we prove the statement for a connected graph, and then it is easy to finish the prove for the disconnected case. That is for each component $G_i$, let $v(G_i)=n_i$. If $e(G_i)\leq\frac{18}{7}n_i-\frac{27}{7}$ holds, clearly we have $e(G)=\sum\limits_{i}e(G_i)\leq\frac{18}{7}n-\frac{27}{7}$.

Let the number of blocks, maximal subgraphs of $G$ containing no cut vertex, be $b$. Let the blocks are $B_1', B_2', B_3',\dots , B_b'$ and with number of vertices (including the cut vertices) $n_1,n_2,\dots, n_b$ respectively. 

It is easy to check that, if $B_i'$ is a block with $2\leq n_i\leq 5$, then $e(B_i')\leq \frac{18}{7}n_i-\frac{27}{7}$. Suppose that $n_i\geq 6$. If there is a vertex of degree 2 in $B_i'$, say $v$, then by induction, $$e(B_i')= e(B'_i-v)+2\leq \frac{18}{7}(n_i-1)-\frac{27}{7}+2=\frac{18}{7}n_i-\frac{31}{7}\leq \frac{18}{7}n_i-\frac{27}{7}.$$
So suppose that $d_{B_i'}(u)\geq 3$ for all $u\in V(B_i')$. By Theorem \ref{vcb}, $$e(B_i')\leq \frac{18}{7}n_i-\frac{48}{7}\leq \frac{18}{7}n_i-\frac{27}{7}.$$
Hence, for each block $B_i'$, $i\in\{1,2,\dots,b\}$, $e(B_i')\leq \dfrac{18}{7}n_i-\dfrac{27}{7}.$ Therefore,
\begin{align*}
e(G)\leq \sum\limits_{i=1}^{b}\bigg(\dfrac{18}{7}n_i-\dfrac{27}{7}\bigg)=\dfrac{18}{7}\sum\limits_{i=1}^{b}n_i-\dfrac{27}{7}b=\dfrac{18}{7}(n+b-1)-\dfrac{27}{7}b=\dfrac{18}{7}n-\frac{9}{7}b-\dfrac{18}{7}
\end{align*}
where we get, $\dfrac{18}{7}n-\dfrac{9}{7}b-\dfrac{18}{7}\leq \dfrac{18}{7}n-\dfrac{27}{7}$, for $b\geq 1$.
\end{proof}


 From the end of the proof of Lemma $3$, we can directly get the following Lemma.

\begin{lemma}\label{blocks} Let $G$ be an $n$-vertex and $\Theta_6$-free connected plane graph with $b$ blocks. Then $e(G) \leq \frac{18}{7}n -\frac{9b +18}{7}$.
\end{lemma}

Note that Lemma~\ref{blocks} implies that if $G$ is an $n$-vertex $\Theta_6$-free planar graph with $b\geq 4$ blocks, then $e(G) \leq \frac{18n}{7} - \frac{54}{7} < \frac{18n}{7} - \frac{48}{7}.$

We need the following claim to prove the lemma which follows.
\begin{claim}\label{3blocks} If $G$ is an $n$-vertex $\Theta_6$-free plane graph containing a vertex of degree 2 such that,  $G-v$ has at least 3 blocks, then $e(G) < \frac{18}{7}n - \frac{48}{7}.$
\end{claim}
\begin{proof}

By Lemma~\ref{blocks} we have that $e(G-v) \leq  \frac{18(n-1)}{7} - \frac{45}{7}$, hence $$e(G) \leq \frac{18n - 18 - 45}{7} + 2 = \frac{18n}{7} - \frac{49}{7} < \frac{18n}{7} - \frac{48}{7}.$$
\end{proof}

\begin{lemma}\label{2bound} Let $G$ be an $n$-vertex $\Theta_6$-free 2-connected plane graph with $\delta(G)=2$, then $$e(G) \leq \frac{18n}{7} - \begin{cases} \frac{38}{7} & \text{ for } n=6, \\ \frac{42}{7}& \text{ for } n=7, \\ \frac{46}{7}& \text{ for } n=8, \\ \frac{48}{7}& \text{ for } n\geq 9. \end{cases}$$
\end{lemma}

\begin{proof}

Let $n=6$, we are going to show that in fact $e(G)\leq \frac{18}{7}n-\frac{38}{7} = 10$ for any $\Theta_6$-free plane graph. 
We claim that  $G$ does not contain $11$ edges. 
Indeed, if $G$ contains $11$ edges, then an embedding of  $G$ on the plane contains a $4$-face and all the remaining faces are of size $3$. 
Thus, fixing the $4$-face the unbounded face of the plane drawing of $G$, then $G$ is a triangular-block on $6$ vertices. 
Thus, $G$ contains a $\Theta_6$, which is a contradiction. 
Therefore, $e(G)\leq 10$. 

Let $n=7$, $G$ be an $n$-vertex $\Theta_6$-free plane graph which is $2$-connected. 
We want to show that $e(G)\leq \frac{18}{7}n-\frac{42}{7}$.  Let $v$ be a vertex of degree $2$ in $G$. 
Thus $e(G-v)\leq \frac{18}{7}(n-1)-\frac{38}{7}$. 
Hence $e(G)=e(G-v)+2\leq \frac{18}{7}(n-1)-\frac{38}{7}+2\leq \frac{18}{7}n-\frac{42}{7}$.

It can be shown that for $n=7$ and any $\Theta_6$-free plane graph $G$,  $e(G)\leq \frac{18}{7}n-\frac{42}{7}.$ 
Indeed, if $G$ is $2$-connected, we have already proved. 
If $G$ contains at least $3$ blocks, then it holds by Lemma \ref{blocks}. 
Suppose that it contains only $2$ blocks. 
Let the blocks be $B_1'$ and $B_2'$ with number of vertices $n_1$ and $n_2$ respectively. 
If $n_1=2$ and $n_2=6$, then $e(B_1')=1$ and $e(B_2)\leq 10$. Hence $e(G)\leq 11$, which implies $e(G)\leq \frac{18}{7}n-\frac{42}{7}$. 
If $n_1=3$ and $n_2=5$, then $e(B_1')=3$ and $e(B_2')\leq 9$. 
Thus, $e(G)\leq 12$, that means $e(G)\leq \frac{18}{7}n-\frac{42}{7}$. 
If $n_1=n_2=4$, then $e(B_1')\leq 6$ and $e(B_2')\leq 6$. Thus, $e(G)\leq 12$, again $e(G)\leq \frac{18}{7}n-\frac{42}{7}.$ 
Therefore from all the results we have, $e(G)\leq \frac{18}{7}n-\frac{42}{7}.$

Let $n=8$, and $G$ be an $n$-vertex $\Theta_6$-free plane graph which is $2$-connected. We want to show that $e(G)\leq \frac{18}{7}n-\frac{46}{7}$. Let $v$ be a vertex of degree $2$ in $G$. Thus, $G-v$ is a  $7$-vertex $\Theta_6$-free plane graph. Thus, from the previous result $e(G-v)\leq \frac{18}{7}(n-1)-\frac{42}{7}.$ Thus, $e(G)\leq \frac{18}{7}(n-1)-\frac{42}{7}+2=\frac{18}{7}n-\frac{46}{7}.$  

We observe further bounds on $e(G)$ if $G$ is not $2$-connected. If $G$ contains two blocks, we claim that $e(G)\leq \frac{18}{7}n-\frac{39}{7}$. Indeed, let the blocks be $B_1'$ and $B_2'$ with $n_1$ and $n_2$ vertices respectively. If $n_1=2$ and $n_2=7$, then $e(B_1')=1$ and $e(B_2')\leq 12$. Thus, $e(G)\leq 13$, that means $e(G)\leq \frac{18}{7}n-\frac{46}{7}.$  If $n_1=3$ and $n_2=6$, then $e(B_1')=3$ and $e(B_2')\leq 10$. Thus, $e(G)\leq 13$.  If $n_1=4$ and $n_2=5$, then $e(B_1')\leq 6$ and $e(B_2')\leq 9$. Thus, $e(G)\leq 15$. Therefore, $e(G)\leq \frac{18}{7}n-\frac{39}{7}.$ Observe that the only structure such that $e(G)>\frac{18}{7}n-\frac{46}{7}$ is when $B_1'$ is a $K_4$ and $B_2'$ is a $K_5^- $ or vice-versa (see Figure \ref{ffr}).

Let $n=9$, and $G$ be an $n$-vertex $\Theta_6$-free plane graph which is $2$-connected. We want to show that $e(G)\leq \frac{18}{7}n-\frac{50}{7}$ and hence $e(G)\leq\frac{18}{7}n-\frac{48}{7}$. Let $v$ be a vertex of degree $2$ in $G$. Thus, $G-v$ is plane graph of $8$ vertices which is $\Theta_6$ free. If $G-v$ is $2$-connected, then from previous results, $e(G-v)\leq \frac{18}{7}(n-1)-\frac{46}{7}.$ Thus, $e(G)\leq\frac{18}{7}(n-1)-\frac{46}{7}+2=\frac{18}{7}n-\frac{50}{7}.$  If $G-v$ has two blocks, then $e(G)\leq \frac{18}{7}(n-1)-\frac{46}{7}$. Indeed, the only case where we get $e(G-v)>\frac{18}{7}(n-1)-\frac{46}{7}$ is when one block is of size $4$ and the other is of size $5$ (see Figure \ref{ffr}). But in that case, there is no possible way to join the two blocks with a cherry (the two endvertices of the degree $2$ vertex belong to $K_{4}$ and $K^{-}_{5}$, respectively). If that is so, it is easy to get a $\Theta_6$ in the graph. Therefore, $e(G)\leq \frac{18}{7}(n-1)-\frac{46}{7}$. Thus, $e(G)\leq \frac{18}{7}(n-1)-\frac{46}{7}+2=\frac{18}{7}n-\frac{50}{7}.$ 
Therefore $e(G)\leq \frac{18}{7}n-\frac{50}{7}.$

\begin{figure}[ht]
\centering
\begin{tikzpicture}[scale=0.15]
\draw[thick](-6,0)--(6,0)--(0,8)--(-6,0);
\draw[thick,red](6,0)--(18,0)--(12,10)--(6,0);
\filldraw[black, fill opacity=0.1] plot coordinates{(-6,0)(6,0)(0,8)(-6,0)};
\filldraw[red, fill opacity=0.1] plot coordinates{(6,0)(18,0)(12,10)(6,0)};
\node at (0,3) {$K_4$};
\node at (12,3) {$K_5^- $};
\draw[fill=black](-6,0)circle(15pt);
\draw[fill=black](6,0)circle(15pt);
\draw[fill=black](18,0)circle(15pt);
\draw[fill=black](12,10)circle(15pt);
\draw[fill=black](0,8)circle(15pt);
\end{tikzpicture}
\caption{Structure of a graph on 8 vertices and containing two blocks of size 4 and 5.}
\label{ffr}
\end{figure}

Let $n=9$ and $G$ be an $n$-vertex $\Theta_6$-free plane graph containing two blocks, say $B_1'$ and $B_2'$ with number of vertices $n_1$ and $n_2$ respectively. If $n_1=2$ and $n_2=8$, then $e(B_1')=1$ and $e(B_2')\leq 14$. Thus, $e(G)\leq 15$. That means, $e(G)\leq \frac{18}{7}n-\frac{50}{7}.$  If $n_1=3$ and $n_2=7$, then $e(B_1')=3$ and $e(B_2')\leq 12$. Again in this case,  $e(G)\leq\frac{18}{7}n-\frac{50}{7}$. If $n_1=4$ and $n_2=6$, then $e(B_1')\leq 6$ and $e(B_2')\leq 10$ and again $e(G)\leq \frac{18}{7}n-\frac{50}{7}$. If $n_1=n_2=5$, then $e(B_1)\leq 9$ and $e(B_2)\leq 9$. If $e(B_1')=e(B_2')=9$, then $e(G)=18$ in this case both $B_1'$ and $B_2'$ are $K_5^- $ and hence the structure of the graph is well known (see Figure \ref{fgr}(a)). The remaining only possibility where we have $e(G)>16$ is when one block contains $9$ edges and the other contains $8$ edges. Without loss of generality, assume $e(B_1')=9$ and $e(B_2')=8$. Thus, $B_1'$ is $K_5^- $. Now we need to figure out the structure of $B_2'$. Notice that $B_2'$ misses only one edge not to be a maximal planar graph with $5$ vertices. We denote the block as $K_5^{--}$. Notice that the plane drawing of $B_2'$ contains one $4$-face and the others are all $3$-face. Thus, if the unbounded face of $B_2'$ is a triangle, then the structure of $G$ is as shown in Figure \ref{fgr}(b). If the unbounded face of $B_2'$ is a $4$-face, then the structure of  $G$ is as shown in Figure \ref{fgr}(c).

Let $n=10$, and $G$ be an $n$-vertex $\Theta_6$-free plane graph which is $2$-connected. We want to show that $e(G)\leq \frac{18}{7}n-\frac{54}{7}$. Let $v$ be a vertex of degree $2$. Thus, $G-v$ is a plane graph on $9$ vertices which is also $\Theta_6$-free. If $G-v$ is $2$-connected, then from the previous results, $e(G-v)\leq \frac{18}{7}(n-1)-\frac{50}{7}$. Thus, $e(G)\leq \frac{18}{7}(n-1)-\frac{50}{7}+2=\frac{18}{7}n-\frac{54}{7}$. If $G-v$ is contains two blocks, then from the previous observation there are three possibilities such that $e(G-v)>\frac{18}{7}(n-1)-\frac{50}{7}$. The structure of these graphs are shown in Figure \ref{fgr}(a,b,c). However, in all the cases if there is a cherry joining the two blocks, then it is easy to get a $\Theta_6$ in the graph. Thus, such 
situations could not appear in $G-v$. This implies that $e(G)\leq 16$. In other words, $e(G-v)\leq \frac{18}{7}(n-1)-\frac{50}{7}$. Therefore, $e(G)\leq \frac{18}{7}n-\frac{50}{7}+2=\frac{18}{7}n-\frac{54}{7}$.


\begin{claim}\label{size6} 
Let $G$ be an $n$-vertex $\Theta_6$-free plane graph with a vertex $v$ of degree $2$, if $G-v$ has exactly $2$ blocks, one with size at least $6$, then $e(G) < \frac{18}{7}n - \frac{48}{7}.$
\end{claim}
\begin{proof}
Let $B'_1$ and $B'_2$ be the blocks of $G$ of sizes $n_1$ and $n_2$ respectively, with $n_1 \geq 6$. 
By induction we may assume that $e(B_1) \leq \frac{18n_1}{7} - \frac{38}{7}$ and by Lemma~\ref{de} $e(B'_2) \leq \frac{18n_2}{7} - \frac{27}{7}$, therefore \begin{displaymath}
e(G) \leq e(B_1') + e(B_2') + 2 \leq \frac{18(n_1+n_2)}{7} - \frac{27+38-14}{7} = \frac{18n}{7} - \frac{51}{7} < \frac{18n}{7} - \frac{48}{7}.\qedhere \end{displaymath}
\end{proof}

Now suppose $n\geq 11$. 
If $\delta(G) \geq 3$, by Theorem~\ref{vcb}  $e(G) \leq \frac{18n}{7} - \frac{48}{7}$.
So suppose there is a vertex $v$ in $G$ with degree $d(v) \leq 2$.
If $G-v$ is $2$-connected, then by Theorem $2$ $e(G) \leq \frac{18(n-1)}{7} - \frac{48}{7} + 2 < \frac{18n}{7} - \frac{48}{7}$. 
If $G-v$ contains precisely $2$ blocks of sizes $n_1, n_2$ with $n_1 \geq n_2$, then since $n_1 + n_2 = n$ we have that $n_1 \geq 6$, and so, by Claim~\ref{size6} we have that $e(G) < \frac{18n}{7} - \frac{48}{7}$. 
If $G-v$ contains at least three blocks, then by Claim~\ref{3blocks} $e(G) < \frac{18n}{7} - \frac{48}{7}$.

\end{proof}

\begin{figure}[ht]
\centering
\begin{tikzpicture}[scale=0.15]
\draw[thick](-6,0)--(6,0)--(0,10)--(-6,0);
\draw[thick,red](6,0)--(18,0)--(12,10)--(6,0);
\filldraw[black, fill opacity=0.1] plot coordinates{(-6,0)(6,0)(0,10)(-6,0)};
\filldraw[red, fill opacity=0.1] plot coordinates{(6,0)(18,0)(12,10)(6,0)};
\draw[fill=black](-6,0)circle(15pt);
\draw[fill=black](6,0)circle(15pt);
\draw[fill=black](18,0)circle(15pt);
\draw[fill=black](12,10)circle(15pt);
\draw[fill=black](0,10)circle(15pt);
\node at (6,-4) {$(a)$};
\node at (0,3) {$K_5^- $};
\node at (12,3) {$K_5^- $};
\end{tikzpicture}\qquad\qquad
\begin{tikzpicture}[scale=0.15]
\draw[thick](-6,0)--(6,0)--(0,10)--(-6,0);
\draw[thick,red](6,0)--(18,0)--(12,10)--(6,0);
\filldraw[black, fill opacity=0.1] plot coordinates{(-6,0)(6,0)(0,10)(-6,0)};
\filldraw[red, fill opacity=0.1] plot coordinates{(6,0)(18,0)(12,10)(6,0)};
\draw[fill=black](-6,0)circle(12pt);
\draw[fill=black](6,0)circle(15pt);
\draw[fill=black](18,0)circle(15pt);
\draw[fill=black](12,10)circle(15pt);
\draw[fill=black](0,10)circle(15pt);
\node at (6,-4) {$(b)$};
\node at (0,3) {$K_5^- $};
\node at (12,3) {$K_5^{--}$};
\end{tikzpicture}\qquad\qquad
\begin{tikzpicture}[scale=0.15]
\draw[thick](-6,0)--(6,0)--(0,10)--(-6,0);
\draw[thick,red](6,0)--(18,0)--(12,10)--(6,0)(18,0)--(24,10)--(12,10)(24,10)--(6,0);
\filldraw[black, fill opacity=0.1] plot coordinates{(-6,0)(6,0)(0,10)(-6,0)};
\draw[fill=black](-6,0)circle(15pt);
\draw[fill=black](6,0)circle(15pt);
\draw[fill=black](18,0)circle(15pt);
\draw[fill=black](24,10)circle(15pt);
\draw[fill=black](12,10)circle(15pt);
\draw[fill=black](0,10)circle(15pt);
\draw[fill=black](15,5)circle(15pt);
\node at (6,-4) {$(c)$};
\node at (0,3) {$K_5^- $};
\end{tikzpicture}
\caption{Structure of graphs on 9 vertices and containing two blocks of size 5.}
\label{fgr}
\end{figure}


\begin{lemma}\label{3bound}
Let $G$ be an $n$-vertex, $\Theta_6$-free plane graph containing a block, say $B$ with $n_b$ vertices, such that $e(B)\leq \frac{18}{7}n_b-\frac{48}{7}$. Then $e(G)\leq \frac{18}{7}n-\frac{48}{7}.$
\end{lemma}
\begin{proof}
Let $G$ contains $b$ blocks. If $b\geq 4$, then by Lemma \ref{blocks}, the statement holds. Suppose that $G$ contains at most $3$ blocks. The number of cut vertices contained in $B$ is at most $2$. We consider two cases:
\begin{enumerate}
   \item $B$ contains no cut vertex. In this case $G=B$ and hence, $e(G)\leq \frac{18}{7}n-\frac{48}{7}.$
    \item $B$ contains only one cut vertex. Let $G'$ be the graph obtained by deleting all vertices of $B$ except the cut vertex. Clearly, $v(G')=n-(n_b-1).$ Using Lemma \ref{de}, $e(G')\leq \frac{18}{7}(n-(n_b-1))-\frac{27}{7}$. Therefore,
     \begin{align*}
      e(G)=e(G')+e(B)\leq \frac{18}{7}(n-(n_b-1))-\frac{27}{7}+\frac{18}{7}n_b-\frac{48}{7}= \frac{18}{7}n-\frac{57}{7}.
    \end{align*}
     Therefore, $e(G)\leq \frac{18}{7}n-\frac{48}{7}.$
   \item $B$ contains two cut vertices. Let $G'$ be the graph obtained by deleting all vertices of $B$  except the two cut vertices. Then $G'$ is a disconnected graph with two components, say $B_1'$ and $B_2'$ with number of vertices $n_1$ and $n_2$ respectively. Clearly $n_1+n_2=n-(n_b-2)$.
    Using Lemma \ref{de}, $e(B_1')\leq \frac{18}{7}n_1-\frac{27}{7}$ and $e(B_2')\leq \frac{18}{7}n_2-\frac{27}{7}$. Hence, 
   \begin{align*}
    e(G')=e(B_1')+e(B_2')&\leq \left(\frac{18}{7}n_1-\frac{27}{7}\right)+\left(\frac{18}{7}n_2-\frac{27}{7}\right)\leq \frac{18}{7}(n_1+n_2)-\frac{54}{7}\\&=\frac{18}{7}(n-(n_b-2))-\frac{54}{7}.   
    \end{align*}
    Thus,
    \begin{align*}
    e(G)=e(G')+e(B)\leq\frac{18}{7}(n-(n_b-2))-\frac{54}{7}+\frac{18}{7}n_b-\frac{48}{7}= \frac{18}{7}n-\frac{66}{7}.
    \end{align*}  
   Therefore, $e(G)\leq \frac{18}{7}n-\frac{48}{7}.$  \qedhere
\end{enumerate}
\end{proof}

Now we give the proof of Theorem \ref{bvc}. Notice that $G$ contains at least $14$ vertices. If $G$ is $2$-connected, then we are done by Lemma \ref{2bound}. Thus we suppose that $G$ contains at least $2$ blocks and at most $3$ blocks. Otherwise, we are done by Lemma \ref{blocks}. So, we distinguish two cases:
\begin{enumerate}
    \item $G$ contains only two blocks. Let the blocks be $B_1'$ and $B_2'$ with number of vertices $n_1$ and $n_2$ respectively, such that $n_1 \geq n_2$. 
    If $n_1\geq 9$, then by Lemma \ref{2bound} and Lemma \ref{3bound}, $e(G)\leq \frac{18}{7}n-\frac{48}{7}.$ 
    
    Since $n+1=n_1+n_2$, we have that if $n\geq 16$ then $n_1 \geq 9$.
    We may assume that either $n= 14$ and $n_1=8,n_2 = 7$ or $n=15$ and $n_1 = 8,n_2=8$.
    
    In the first case  by Lemma \ref{2bound}, $e(B_1')\leq \frac{18}{7}n_2-\frac{46}{7}$ and $e(B_2')\leq \frac{18}{7}n_1-\frac{42}{7}$. 
    Therefore, 
    \begin{align*}
        e(G)=e(B_1')+e(B_2')&\leq \frac{18}{7}(n_1+n_2)-\frac{88}{7}\\&=\frac{18}{7}(n+1)-\frac{88}{7}=\frac{18}{7}n-\frac{70}{7}< \frac{18}{7}n-\frac{48}{7}.
    \end{align*}
    
    In the second case by Lemma \ref{2bound}, $e(B_1')\leq \frac{18}{7}n_2-\frac{46}{7}$ and $e(B_2')\leq \frac{18}{7}n_2-\frac{46}{7}$. 
    Therefore, 
    \begin{align*}
        e(G)=e(B_1')+e(B_2')&\leq \frac{18}{7}(n_1+n_2)-\frac{92}{7}\\&=\frac{18}{7}(n+1)-\frac{92}{7}=\frac{18}{7}n-\frac{74}{7}< \frac{18}{7}n-\frac{48}{7}.
    \end{align*}
    
    \item $G$ contains three blocks. Let the blocks be $B_1', B_2'$ and $B_3'$ with number of vertices $n_1, n_2$ and $n_3$ respectively, such that $n_1 \geq n_2 \geq n_3$. Since $n+2=n_1+n_2+n_3$, and $n\geq 14$ we have that $n_1 \geq 6$, hence,  
     by Lemma \ref{2bound}, $e(B_1')\leq \frac{18}{7}n_1-\frac{38}{7}$. 
     From Lemma \ref{de}, $e(B_i')\leq \frac{18}{7}n_2-\frac{27}{7}$, $i\in \{2,3\}$. 
     Thus, 
    \begin{align*}
    e(G)=e(B_1')+e(B_2')+e(B_3')&\leq\frac{18}{7}(n_1+n_2+n_3)-\frac{27}{7}-\frac{27}{7}-\frac{38}{7}\\&=\frac{18}{7}(n+2)-\frac{92}{7}=\frac{18}{7}n-\frac{56}{7}\\&< \frac{18}{7}n-\frac{48}{7}.
    \end{align*}
\end{enumerate}
Notice that for $n=13$, we have a counter example for which Theorem \ref{bvc} does not hold. One counter example is the graph $G$ shown in Figure \ref{dfc}, where the graph contains $27$ edges but $e(G)>\frac{18}{7}n-\frac{48}{7}.$
\begin{figure}[ht]
\centering
\begin{tikzpicture}[scale=0.2]
\draw[fill=black](0,0)circle(12pt);
\draw[fill=black](0,5)circle(12pt);
\draw[fill=black](0,10)circle(12pt);
\draw[fill=black](-6,-4)circle(12pt);
\draw[fill=black](6,-4)circle(12pt);
\draw[fill=black](12,0)circle(12pt);
\draw[fill=black](12,5)circle(12pt);
\draw[fill=black](12,10)circle(12pt);
\draw[fill=black](6,-4)circle(12pt);
\draw[fill=black](18,-4)circle(12pt);
\draw[fill=black](24,0)circle(12pt);
\draw[fill=black](24,5)circle(12pt);
\draw[fill=black](24,10)circle(12pt);
\draw[fill=black](30,-4)circle(12pt);
\draw[fill=black](18,-4)circle(12pt);
\draw[thick](0,0)--(0,5)--(0,10)--(6,-4)--(-6,-4)--(0,10)(-6,-4)--(0,0)--(6,-4)--(0,5)--(-6,-4);
\draw[thick](12,0)--(12,5)--(12,10)--(6,-4)--(18,-4)--(12,10)(18,-4)--(12,0)--(6,-4)--(12,5)--(18,-4);
\draw[thick](24,0)--(24,5)--(24,10)--(30,-4)--(18,-4)--(24,10)(18,-4)--(24,0)--(24,10)--(24,5)--(30,-4)(18,-4)--(24,5)(30,-4)--(24,0);
\end{tikzpicture}
\caption{Maximal counter example}
\label{dfc}
\end{figure}
\section{Concluding Remarks and Conjectures}
As we know, $\Theta_k$ is the family of Theta graphs on $k$ vertices, that is, graphs obtained from the $k$-cycle by adding an additional edge joining two non-consecutive vertices. In particular, $\Theta_6=\{\Theta_6^1, \Theta_6^2\}$, where $\Theta_6^1$ and $\Theta_6^2$  are the symmetric and asymmetric $\Theta_6$-graphs which are shown in Figure \ref{teta}(left) and Figure \ref{teta}(right) respectively. One may ask what the planar Tur\'an number of $\Theta_6^1$ and $\Theta_6^2$ are. What can be said about the tight upper bounds of   $\text{ex}_{\mathcal{P}}(n, \Theta_6^1)$ and $\text{ex}_{\mathcal{P}}(n,\Theta_6^2)$? 

We pose the following asymptotic conjectures.
\newpage
\begin{conjecture}\
\begin{enumerate}
    \item $\ex_{\mathcal{P}}(n,\Theta_6^1)= \frac{45}{17}n+\Theta(1).$ 
    \item $\ex_{\mathcal{P}}(n,\Theta_6^2)= \frac{18}{7}n+\Theta(1).$ 
\end{enumerate}
\end{conjecture}
The lower bound for $\text{ex}_{\mathcal{P}}(n, \Theta_6^1)$ is a recursive construction. The base construction is obtained by identifying the outer pentagon and the boundary of the shaded inner pentagon of Figure \ref{fgv}(a) and (b) respectively. Each of the triangular region (shaded region) in the construction is a $K_5^{-}$. Notice that, every non-triangular face in the construction is of size 5 and is surrounded by five $K_5^{-}$'s. It can be checked that the graph contains no $\Theta_6^1$ and has as many edges as indicated in the bound.  
\begin{figure}[ht]
\centering
\begin{tikzpicture}[scale=0.15]

\draw[ultra thick](8.6,5)--(0,10)--(-8.6,5)--(-8.6,-5)--(0,-10)--(8.6,-5)--(8.6,5);
\draw[ultra thick,red](8.6,5)--(-8.6,-5);
\draw[fill=black](8.6,5)circle(18pt);
\draw[fill=black](0,10)circle(18pt);
\draw[fill=black](-8.6,5)circle(18pt);
\draw[fill=black](-8.6,-5)circle(18pt);
\draw[fill=black](0,-10)circle(18pt);
\draw[fill=black](8.6,-5)circle(18pt);
\node at (0,-15) {$\Theta_6^1$};
\end{tikzpicture}\qquad\qquad\qquad
\begin{tikzpicture}[scale=0.15]
\draw[ultra thick](8.6,5)--(0,10)--(-8.6,5)--(-8.6,-5)--(0,-10)--(8.6,-5)--(8.6,5);
\draw[ultra thick,red](-8.6,-5)--(0,10);
\draw[fill=black](8.6,5)circle(18pt);
\draw[fill=black](0,10)circle(18pt);
\draw[fill=black](-8.6,5)circle(18pt);
\draw[fill=black](-8.6,-5)circle(18pt);
\draw[fill=black](0,-10)circle(18pt);
\draw[fill=black](8.6,-5)circle(18pt);
\node at (0,-15) {$\Theta_6^2$};
\end{tikzpicture}
\caption{$\Theta_6$-graphs}
\label{teta}
\end{figure}

\begin{figure}[h]
\centering
\begin{tikzpicture}[scale=0.06]
\begin{scope}[yshift=-7cm,rotate=180]
\filldraw[red, fill opacity=0.1] plot coordinates{(9.5,3.1)(8.8,12.1)(0,10)(9.5,3.1)};
\filldraw[red, fill opacity=0.1] plot coordinates{(-9.5,3.1)(-8.8,12.1)(0,10)(-9.5,3.1)};
\filldraw[red, fill opacity=0.1] plot coordinates{(9.5,3.1)(14.3,-4.6)(5.9,-8.1)(9.5,3.1)};
\filldraw[red, fill opacity=0.1] plot coordinates{(-9.5,3.1)(-14.3,-4.6)(-5.9,-8.1)(-9.5,3.1)};
\filldraw[red, fill opacity=0.1] plot coordinates{(5.9,-8.1)(0,-15)(-5.9,-8.1)(5.9,-8.1)};
\filldraw[red, fill opacity=0.1] plot coordinates{(16.2,11.8)(8.8,12.1)(6.2,19)(16.2,11.8)};
\filldraw[red, fill opacity=0.1] plot coordinates{(-16.2,11.8)(-8.8,12.1)(-6.2,19)(-16.2,11.8)};
\filldraw[red, fill opacity=0.1] plot coordinates{(-20,0)(-14.3,-4.6)(-16.2,-11.8)(-20,0)};
\filldraw[red, fill opacity=0.1] plot coordinates{(20,0)(14.3,-4.6)(16.2,-11.8)(20,0)};
\filldraw[red, fill opacity=0.1] plot coordinates{(6.2,19)(0,30)(-6.2,19)(6.2,19)};
\filldraw[red, fill opacity=0.1] plot coordinates{(-20,0)(-28.5,9.3)(-16.2,11.8)(-20,0)};
\filldraw[red, fill opacity=0.1] plot coordinates{(20,0)(28.5,9.3)(16.2,11.8)(20,0)};
\filldraw[red, fill opacity=0.1] plot coordinates{(-6.2,-19)(0,-15)(6.2,-19)(-6.2,-19)};
\filldraw[red, fill opacity=0.1] plot coordinates{(-16.2,-11.8)(-17.6,-24.3)(-6.2,-19)(-16.2,-11.8)};
\filldraw[red, fill opacity=0.1] plot coordinates{(16.2,-11.8)(17.6,-24.3)(6.2,-19)(16.2,-11.8)};
\filldraw[red, fill opacity=0.1] plot coordinates{(26.5,36.4)(0,30)(-26.5,36.4)(26.5,36.4)};
\filldraw[red, fill opacity=0.1] plot coordinates{(-26.5,36.4)(-28.5,9.3)(-42.8,-13.9)(-26.5,36.4)};
\filldraw[red, fill opacity=0.1] plot coordinates{(26.5,36.4)(28.5,9.3)(42.8,-13.9)(26.5,36.4)};
\filldraw[red, fill opacity=0.1] plot coordinates{(-42.8,-13.9)(-17.6,-24.3)(0,-45)(-42.8,-13.9)};
\filldraw[red, fill opacity=0.1] plot coordinates{(42.8,-13.9)(17.6,-24.3)(0,-45)(42.8,-13.9)};
\filldraw[ultra thick,blue,pattern=dots,pattern color=blue](9.5,3.1)--(0,10)--(-9.5,3.1)--(-5.9,-8.1)--(5.9,-8.1)--(9.5,3.1);
\draw[thick](9.5,3.1)--(8.8,12.1)--(0,10);
\draw[thick](-9.5,3.1)--(-8.8,12.1)--(0,10);
\draw[thick](9.5,3.1)--(14.3,-4.6)--(5.9,-8.1);
\draw[thick](-9.5,3.1)--(-14.3,-4.6)--(-5.9,-8.1);
\draw[thick](5.9,-8.1)--(0,-15)--(-5.9,-8.1);
\draw[thick](20,0)--(16.2,11.8)--(6.2,19)--(-6.2,19)--(-6.2,19)--(-16.2,11.8)--(-20,0)--(-16.2,-11.8)--(-6.2,-19)--(6.2,-19)--(16.2,-11.8)--(20,0);
\draw[thick](16.2,11.8)--(8.8,12.1)--(6.2,19);
\draw[thick](-16.2,11.8)--(-8.8,12.1)--(-6.2,19);
\draw[thick](-20,0)--(-14.3,-4.6)--(-16.2,-11.8);
\draw[thick](20,0)--(14.3,-4.6)--(16.2,-11.8);
\draw[thick](-6.2,-19)--(0,-15)--(6.2,-19);
\draw[thick](20,0)--(28.5,9.3)--(16.2,11.8);
\draw[thick](6.2,19)--(0,30)--(-6.2,19);
\draw[thick](-20,0)--(-28.5,9.3)--(-16.2,11.8);
\draw[thick](-16.2,-11.8)--(-17.6,-24.3)--(-6.2,-19);
\draw[thick](16.2,-11.8)--(17.6,-24.3)--(6.2,-19);
\draw[thick](28.5,9.3)--(26.5,36.4)--(0,30)--(-26.5,36.4)--(-28.5,9.3)--(-42.8,-13.9)--(-17.6,-24.3)--(0,-45)--(17.6,-24.3)--(42.8,-13.9)--(28.5,9.3);
\draw[ultra thick, red](26.5,36.4)--(-26.5,36.4)--(-42.8,-13.9)--(0,-45)--(42.8,-13.9)--(26.5,36.4);
\draw[fill=black](0,10)circle(40pt);
\draw[fill=black](9.5,3.1)circle(40pt);
\draw[fill=black](-9.5,3.1)circle(40pt);
\draw[fill=black](5.9,-8.1)circle(40pt);
\draw[fill=black](-5.9,-8.1)circle(40pt);
\draw[fill=black](8.8,12.1)circle(40pt);
\draw[fill=black](-8.8,12.1)circle(40pt);
\draw[fill=black](14.3,-4.6)circle(40pt);
\draw[fill=black](-14.3,-4.6)circle(40pt);
\draw[fill=black](0,-15)circle(40pt);
\draw[fill=black](20,0)circle(40pt);
\draw[fill=black](16.2,11.8)circle(40pt);
\draw[fill=black](6.2,19)circle(40pt);
\draw[fill=black](-6.2,19)circle(40pt);
\draw[fill=black](-16.2,11.8)circle(40pt);
\draw[fill=black](-20,0)circle(40pt);
\draw[fill=black](-16.2,-11.8)circle(40pt);
\draw[fill=black](-6.2,-19)circle(40pt);
\draw[fill=black](6.2,-19)circle(40pt);
\draw[fill=black](16.2,-11.8)circle(40pt);
\draw[fill=black](28.5,9.3)circle(40pt);
\draw[fill=black](0,30)circle(40pt);
\draw[fill=black](-28.5,9.3)circle(40pt);
\draw[fill=black](-17.6,-24.3)circle(40pt);
\draw[fill=black](17.6,-24.3)circle(40pt);
\draw[fill=black](26.5,36.4)circle(40pt);
\draw[fill=black](-26.5,36.4)circle(40pt);
\draw[fill=black](-42.8,-13.9)circle(40pt);
\draw[fill=black](42.8,-13.9)circle(40pt);
\draw[fill=black](0,-45)circle(40pt);
\end{scope}
\node at (0,-50) {(a)};
\end{tikzpicture}\qquad
\begin{tikzpicture}[scale=0.06]
\filldraw[red, fill opacity=0.1] plot coordinates{(-7.3,-34)(0,-45)(7.3,-34)(-7.3,-34)};
\filldraw[red, fill opacity=0.1] plot coordinates{(-34.8,-3.7)(-42.8,-13.9)(-30.3,-17.5)(-34.8,-3.7)};
\filldraw[red, fill opacity=0.1] plot coordinates{(34.8,-3.7)(42.8,-13.9)(30.3,-17.5)(34.8,-3.7)};
\filldraw[red, fill opacity=0.1] plot coordinates{(-7.3,-34)(0,-45)(7.3,-34)(-7.3,-34)};
\filldraw[red, fill opacity=0.1] plot coordinates{(14.2,32)(26.5,36.4)(26,23.4)(14.2,32)};
\filldraw[red, fill opacity=0.1] plot coordinates{(-14.2,32)(-26.5,36.4)(-26,23.4)(-14.2,32)};
\filldraw[red, fill opacity=0.1] plot coordinates{(33.3,10.8)(26,23.4)(18.6,16.7)(33.3,10.8)};
\filldraw[red, fill opacity=0.1] plot coordinates{(14.2,32)(0,35)(10.2,23)(14.2,32)};
\filldraw[red, fill opacity=0.1] plot coordinates{(-33.3,10.8)(-26,23.4)(-18.6,16.7)(-33.3,10.8)};
\filldraw[red, fill opacity=0.1] plot coordinates{(-14.2,32)(0,35)(-10.2,23)(-14.2,32)};
\filldraw[red, fill opacity=0.1] plot coordinates{(-33.3,10.8)(-34.8,-3.7)(-24.9,-2.7)(-33.3,10.8)};
\filldraw[red, fill opacity=0.1] plot coordinates{(-30.3,-17.5)(-21.7,-12.5)(-20.6,-28.3)(-30.3,-17.5)};
\filldraw[red, fill opacity=0.1] plot coordinates{(-20.6,-28.3)(-7.3,-34)(-5.2,-24.5)(-20.6,-28.3)};
\filldraw[red, fill opacity=0.1] plot coordinates{(33.3,10.8)(34.8,-3.7)(24.9,-2.7)(33.3,10.8)};
\filldraw[red, fill opacity=0.1] plot coordinates{(30.3,-17.5)(21.7,-12.5)(20.6,-28.3)(30.3,-17.5)};
\filldraw[red, fill opacity=0.1] plot coordinates{(20.6,-28.3)(7.3,-34)(5.2,-24.5)(20.6,-28.3)};
\filldraw[red, fill opacity=0.1] plot coordinates{(11.1,10)(18.6,16.7)(11.8,16.2)(11.1,10)};
\filldraw[red, fill opacity=0.1] plot coordinates{(6.1,13.7)(11.8,16.2)(10.2,23)(6.1,13.7)};
\filldraw[red, fill opacity=0.1] plot coordinates{(-11.1,10)(-18.6,16.7)(-11.8,16.2)(-11.1,10)};
\filldraw[red, fill opacity=0.1] plot coordinates{(-6.1,13.7)(-11.8,16.2)(-10.2,23)(-6.1,13.7)};
\filldraw[red, fill opacity=0.1] plot coordinates{(-24.9,-2.7)(-14.9,-1.6)(-19,-6.2)(-24.9,-2.7)};
\filldraw[red, fill opacity=0.1] plot coordinates{(-19,-6.2)(-13,-7.5)(-21.7,-12.5)(-19,-6.2)};
\filldraw[red, fill opacity=0.1] plot coordinates{(24.9,-2.7)(14.9,-1.6)(19,-6.2)(24.9,-2.7)};
\filldraw[red, fill opacity=0.1] plot coordinates{(19,-6.2)(13,-7.5)(21.7,-12.5)(19,-6.2)};
\filldraw[red, fill opacity=0.1] plot coordinates{(-5.2,-24.5)(-3.1,-14.7)(0,-20)(-5.2,-24.5)};
\filldraw[red, fill opacity=0.1] plot coordinates{(5.2,-24.5)(3.1,-14.7)(0,-20)(5.2,-24.5)};
\filldraw[red, fill opacity=0.1] plot coordinates{(14.9,-1.6)(11.1,10)(9.5,3.1)(14.9,-1.6)};
\filldraw[red, fill opacity=0.1] plot coordinates{(-6.1,13.7)(6.1,13.7)(0,10)(-6.1,13.7)};
\filldraw[red, fill opacity=0.1] plot coordinates{(-14.9,-1.6)(-11.1,10)(-9.5,3.1)(-14.9,-1.6)};
\filldraw[red, fill opacity=0.1] plot coordinates{(-13,-7.5)(-3.1,-14.7)(-5.9,-8.1)(-13,-7.5)};
\filldraw[red, fill opacity=0.1] plot coordinates{(13,-7.5)(3.1,-14.7)(5.9,-8.1)(13,-7.5)};
\draw[ultra thick,blue](26.5,36.4)--(-26.5,36.4)--(-42.8,-13.9)--(0,-45)--(42.8,-13.9)--(26.5,36.4);
\draw[thick](-34.8,-3.7)--(-42.8,-13.9)--(-30.3,-17.5);
\draw[thick](34.8,-3.7)--(42.8,-13.9)--(30.3,-17.5);
\draw[thick](-7.3,-34)--(0,-45)--(7.3,-34);
\draw[thick](-20.6,-28.3)--(-7.3,-34)--(-5.2,-24.5)(20.6,-28.3)--(7.3,-34)--(5.2,-24.5)(-7.3,-34)--(7.3,-34);
\draw[thick](33.3,10.8)--(34.8,-3.7)--(24.9,-2.7)(20.6,-28.3)--(30.3,-17.5)--(21.7,-12.5)(30.3,-17.5)--(34.8,-3.7);
\draw[thick](-33.3,10.8)--(-34.8,-3.7)--(-24.9,-2.7)(-20.6,-28.3)--(-30.3,-17.5)--(-21.7,-12.5)(-30.3,-17.5)--(-34.8,-3.7);
\draw[thick](-10.2,23)--(-14.2,32)--(0,35)(-33.3,10.8)--(-26,23.4)--(-18.6,16.7)(-14.2,32)--(-26,23.4);
\draw[thick](10.2,23)--(14.2,32)--(0,35)(33.3,10.8)--(26,23.4)--(18.6,16.7)(14.2,32)--(26,23.4);
\draw[thick](-19,-6.2)--(-21.7,-12.5)--(-13,-7.5)(-3.1,-14.7)--(-5.2,-24.5)--(0,-20)(19,-6.2)--(21.7,-12.5)--(13,-7.5)(3.1,-14.7)--(5.2,-24.5)--(0,-20);
\draw[thick](-11.8,16.2)--(-10.2,23)--(-6.1,13.7)(11.8,16.2)--(10.2,23)--(6.1,13.7);
\draw[thick](-11.8,16.2)--(-18.6,16.7)--(-11.1,10)(-19,-6.2)--(-24.9,-2.7)--(-14.9,-1.6);
\draw[thick](11.8,16.2)--(18.6,16.7)--(11.1,10)(19,-6.2)--(24.9,-2.7)--(14.9,-1.6);
\draw[thick](14.9,-1.6)--(11.1,10)(6.1,13.7)--(-6.1,13.7)(-13,-7.5)--(-3.1,-14.7)(3.1,-14.7)--(13,-7.5)(-14.9,-1.6)--(-11.1,10);
\draw[thick](-5.9,-8.1)--(-3.1,-14.7)--(0,-20)--(3.1,-14.7)--(5.9,-8.1);
\draw[thick](9.5,3.1)--(14.9,-1.6)--(19,-6.2)--(13,-7.5)--(5.9,-8.1);
\draw[thick](-9.5,3.1)--(-14.9,-1.6)--(-19,-6.2)--(-13,-7.5)--(-5.9,-8.1);
\draw[thick](-9.5,3.1)--(-11.1,10)--(-11.8,16.2)--(-6.1,13.7)--(0,10);
\draw[thick](9.5,3.1)--(11.1,10)--(11.8,16.2)--(6.1,13.7)--(0,10);
\draw[ultra thick,red,pattern=dots,pattern color=red](9.5,3.1)--(0,10)--(-9.5,3.1)--(-5.9,-8.1)--(5.9,-8.1)--(9.5,3.1);
\draw[fill=black](9.5,3.1)circle(40pt);
\draw[fill=black](0,10)circle(40pt);
\draw[fill=black](-9.5,3.1)circle(40pt);
\draw[fill=black](-5.9,-8.1)circle(40pt);
\draw[fill=black](5.9,-8.1)circle(40pt);
\draw[fill=black](9.5,3.1)circle(40pt);
\draw[fill=black](11.1,10)circle(40pt);
\draw[fill=black](11.8,16.2)circle(40pt);
\draw[fill=black](6.1,13.7)circle(40pt);
\draw[fill=black](-11.1,10)circle(40pt);
\draw[fill=black](-11.8,16.2)circle(40pt);
\draw[fill=black](-6.1,13.7)circle(40pt);
\draw[fill=black](-14.9,-1.6)circle(40pt);
\draw[fill=black](-19,-6.2)circle(40pt);
\draw[fill=black](-13,-7.5)circle(40pt);
\draw[fill=black](14.9,-1.6)circle(40pt);
\draw[fill=black](19,-6.2)circle(40pt);
\draw[fill=black](13,-7.5)circle(40pt);
\draw[fill=black](-3.1,-14.7)circle(40pt);
\draw[fill=black](0,-20)circle(40pt);
\draw[fill=black](3.1,-14.7)circle(40pt);
\draw[fill=black](18.6,16.7)circle(40pt);
\draw[fill=black](24.9,-2.7)circle(40pt);
\draw[fill=black](-18.6,16.7)circle(40pt);
\draw[fill=black](-24.9,-2.7)circle(40pt);
\draw[fill=black](-10.2,23)circle(40pt);
\draw[fill=black](10.2,23)circle(40pt);
\draw[fill=black](-21.7,-12.5)circle(40pt);
\draw[fill=black](-5.2,-24.5)circle(40pt);
\draw[fill=black](21.7,-12.5)circle(40pt);
\draw[fill=black](5.2,-24.5)circle(40pt);
\draw[fill=black](33.3,10.8)circle(40pt);
\draw[fill=black](0,35)circle(40pt);
\draw[fill=black](-33.3,10.8)circle(40pt);
\draw[fill=black](-20.6,-28.3)circle(40pt);
\draw[fill=black](20.6,-28.3)circle(40pt);
\draw[thick](18.6,16.7)--(33.3,10.8)--(24.9,-2.7);
\draw[thick](-10.2,23)--(0,35)--(10.2,23);
\draw[thick](-18.6,16.7)--(-33.3,10.8)--(-24.9,-2.7);
\draw[thick](-21.7,-12.5)--(-20.6,-28.3)--(-5.2,-24.5);
\draw[thick](21.7,-12.5)--(20.6,-28.3)--(5.2,-24.5);
\draw[fill=black](26,23.4)circle(40pt);
\draw[fill=black](14.2,32)circle(40pt);
\draw[fill=black](-26,23.4)circle(40pt);
\draw[fill=black](-14.2,32)circle(40pt);
\draw[fill=black](-34.8,-3.7)circle(40pt);
\draw[fill=black](-30.3,-17.5)circle(40pt);
\draw[fill=black](34.8,-3.7)circle(40pt);
\draw[fill=black](30.3,-17.5)circle(40pt);
\draw[fill=black](-7.3,-34)circle(40pt);
\draw[fill=black](7.3,-34)circle(40pt);
\draw[fill=black](26.5,36.4)circle(40pt);
\draw[thick](14.2,32)--(26.5,36.4)--(26,23.4);
\draw[fill=black](-26.5,36.4)circle(40pt);
\draw[thick](-14.2,32)--(-26.5,36.4)--(-26,23.4);
\draw[fill=black](-42.8,-13.9)circle(40pt);
\draw[fill=black](42.8,-13.9)circle(40pt);
\draw[fill=black](0,-45)circle(40pt);
\node at (0,-50) {(b)};
\end{tikzpicture} 
\caption{$\Theta_6^1$-free planar graphs}
\label{fgv}
\end{figure}\newpage
\section*{Acknowledgments}

The authors would like to thank the two anonymous referees for their valuable comments and suggestions which significantly improved the paper.
\section*{Dedications}
\textbf{Addisu Paulos} wants to dedicate this work to all Ethiopians who have been marginalized and made to cast their gaze downward from the heights of Ethiopianism due to the prevailing \textbf{ethnically-framed} regime and politics.

\end{document}